\documentclass[10pt,a4paper]{article}
\usepackage[utf8]{inputenc}

\usepackage{amsfonts}
\usepackage{amssymb}
\usepackage{amsthm}

\usepackage[T1]{fontenc}
\usepackage{newpxtext,newpxmath}

\usepackage{mathrsfs}  

\usepackage[left=2cm,right=2cm,top=2cm,bottom=2cm]{geometry}

\usepackage{hyperref}
\usepackage{mathtools}
\mathtoolsset{showonlyrefs}

\usepackage{enumerate}

\theoremstyle{plain}
\newtheorem{thm}{Theorem}
\newtheorem{lemma}[thm]{Lemma}
\newtheorem{prop}[thm]{Proposition}
\newtheorem{cor}[thm]{Corollary}
\newtheorem{problem}[thm]{Problem}

\theoremstyle{remark}
\newtheorem{remark}[thm]{Remark}
\newtheorem{hyp}[thm]{Assumption}

\theoremstyle{definition}

\newtheorem{defn}[thm]{Definition}

\newcommand{\ov}{\overline}

\usepackage{bbm}

\newcommand{\one}{\mathbbm{1}}

\usepackage{color}
\usepackage{xcolor}

\newcommand{\RR}{\mathbb{R}}
\newcommand{\R}{\mathbb{R}}

\newcommand{\e}{\epsilon}

\newcommand{\mc}{\mathcal}
\newcommand{\mb}{\mathbb}
\newcommand{\ms}{\mathscr}

\newcommand{\be}{\begin{equation}}
\newcommand{\ee}{\end{equation}}

\newcommand*\dd{\mathop{}\!\mathrm{d}}

\renewcommand{\div}{\operatorname{div}}
\renewcommand{\dim}{\operatorname{dim}}
\newcommand{\osc}{\operatorname*{osc}}

\newcommand{\Image}{\operatorname{Im}}

\newcommand{\diam}{\operatorname{diam}}
\newcommand{\Rank}{\operatorname{Rank}}

\numberwithin{thm}{section}
\numberwithin{equation}{section}

\newcommand{\ds}{\displaystyle}

\title{H\"older continuity for non-coercive Hamilton--Jacobi equations associated to linear control systems}

\begin{document}
\author{Megan Griffin-Pickering\thanks{Institute of Mathematics, University of Z\"urich, Winterthurerstrasse 190, 8057 Z\"urich, Switzerland; megan.griffin-pickering@math.uzh.ch} \and Alp\'ar R. M\'esz\'aros\thanks{Department of Mathematical Sciences, Durham University, DH1 3LE Durham, UK; alpar.r.meszaros@durham.ac.uk}}

\maketitle

\begin{abstract}
In this paper we establish H\"older continuity estimates for viscosity solutions to first order Hamilton--Jacobi equations linked to linear control systems satisfying the Kalman rank condition. Our model Hamiltonians are non-convex in the generalised momentum variable and -- more importantly -- they lack coercivity in certain directions. Therefore, all previously available results from the literature cannot be applied to these degenerate settings. In order to overcome these obstructions, we design a geometric argument, dictated by the linear control system. As a result of this, the obtained H\"older estimates are quantified in an anisotropic way within this geometric framework. The estimates hold true for unbounded source terms, for which one part of our analysis is inspired by a recent result on De Giorgi type methods for hypoelliptic operators.
\end{abstract}

\section{Introduction}
The theory of viscosity solutions for Hamilton--Jacobi (HJ) partial differential equations (PDEs) is a fundamental pillar in the modern theory of nonlinear PDEs, since the influential works of Crandall, Evans and Lions (\cite{CraLio, CraEvaLio}), and going back to the works of Kru\v zkov (\cite{Kru:70, Kru:75}). For a comprehensive account of the development of the theory we refer to the monographs \cite{Bardi-CD, CanSin, Tra}.

While rich properties of viscosity solutions --- such as semi-concavity, Lipschitz continuity, structure of singularity sets, etc. --- have been established when a control theoretic formulation is present (and the resulting Hamiltonians are thus convex in the generalised momentum variable) and Hamiltonians are time independent and regular, much less is known about fine regularity properties for models driven by degenerate noise, stochastic forcing or simply when the Hamiltonians are time dependent or lack regularity. Thus, the following natural question may be formulated: can one obtain quantified locally uniform time-space continuity estimates on solutions, independently of convexity and regularity properties of potentially time dependent Hamiltonians?

Motivated by this question, in the past two decades or so the search for qualitative and quantitative H\"older regularity estimates for viscosity solutions to HJ equations became an important research topic in the field, resulting in some breakthrough results. For a non-exhaustive list of works we refer to \cite{Cardaliaguet2009, CardaliaguetGraber, CarRai, CanCar, Cardaliaguet-Silvestre, ChanVasseur, CarSee, ChanVasseur, StoVas, Cir}. These consider typically a mixture of first and second order models. 

A striking outcome of this line of work is that it is now known that H\"older regularity holds for viscosity solutions of first-order HJ equations without regularity or convexity assumptions on the Hamiltonian, imposing instead a \emph{superlinear growth} condition with respect to the momentum variable.
Thus {\it full coercivity} (i.e. uniform coercivity in all directions) can serve as a regularisation mechanism, without the need for elliptic/parabolic structure -- for example, as \cite{Cardaliaguet-Silvestre} formulates: ``The regularisation effect is based only on the strong coercivity assumption of $H$ with respect to $Du$, [\dots].''

However, many interesting HJ equations satisfy only a \emph{degenerate} version of such a coercivity condition.
A large, natural class of examples is provided by Hamilton--Jacobi--Bellman (HJB) equations for optimal control problems in which \emph{certain directions are forbidden} to the controller.
Consider, for example, a linear time-invariant (LTI) control system (of the kind standard in control theory), reading as follows for the state-control pair $(\eta,\beta) \in \R^N \times \R^N$:
\begin{equation}\label{intro:control_sys0}
\dot\eta_{t} = A\eta_{t} + P_{0}\beta_{t}.
\end{equation}
Here $A$ and $P_0$ are constant $N \times N$ matrices. When $\Rank P_0 < N$, the choice of $\dot\eta_{t}$ is restricted to a proper affine subspace of $\R^N$.
A standard form for an optimal control problem driven by \eqref{intro:control_sys0} is to minimise a functional
\be \label{eq:OptCtrl_Functional}
\eta\mapsto \int_0^T L(t, \eta_t, \beta_t) \dd t + g(\eta_T),
\ee
where $L:(0,T)\times\R^N\times \R^N\to\R$ and $g:\R^N\to\R$ are some given Lagrangian and final cost functions, respectively.
The Hamiltonian $\mc{H}$ of the corresponding HJB equation 
\be \label{intro:HJB}
- \partial_t u + \mc{H}(t,x, \nabla_x u) = 0 ,\qquad u(T,x) = g(x)
\ee
then cannot be coercive: since
\be
 \mc{H} \left (t,x, p \right ) : = - \langle Ax ,  p \rangle + \sup_{\beta \in \R^N} \left \{ - \langle P_0 \beta,  p \rangle - L(t,x, \beta) \right \}
\ee
(see \cite{Bardi-CD} for the derivation), consequently for all $p \in \left (P_0 \R^N \right)^\perp$ we have $\mc{H} \left (t,x, p \right ) = - \langle Ax ,  p \rangle - \inf_{\beta \in \R^N} L(t,x, \beta) $, such that $\lim_{|p| \to + \infty} \mc{H} \left (t,x, p \right )= + \infty$ cannot hold. 

On the other hand, for these Hamiltonians it still makes sense to ask for coercivity (or indeed superlinear growth) in the \emph{controlled} directions $p \in P_0 \R^N$. This raises the following question: {\it can H\"older regularity properties be obtained, when the Hamiltonian is coercive only in a limited set of directions?} 
In this paper we answer this question in the affirmative for a rich class of model problems, whose key structural feature is that the drift matrix $A$ is such that \eqref{intro:control_sys0} is \emph{controllable}.

Namely, we establish local H\"older continuity properties of viscosity solutions $u$ to first order HJ equations of the form
\begin{align} \label{eq:HJ_intro}
\partial_t u(t,x) + \left \langle Ax, \nabla_x u(t,x) \right \rangle + H(t, x, \nabla_x u(t,x)) - f(t,x) = 0,
\end{align}
posed on an open subset $\ms{U}$ of $\R\times\R^N$. This choice of structure is motivated by optimal control problems driven by linear control systems of the form \eqref{intro:control_sys0}, but our hypotheses allow more general settings: in particular, we do not require convexity of $H$ in its last variable. 

As standing assumptions on the data, we impose: 

\begin{hyp} \label{hyp:standing}
\begin{enumerate}[(i)]
\item \label{hyp:standing_H} $H$ satisfies the following {\it degenerate} superlinear coercivity condition: there exist an exponent $q>1$ and constants $0 < \lambda \leq  \Lambda < + \infty$ such that
\be\label{ineq:coerc}
\frac{\lambda^q}{q} |P_0 \xi |^q \leq H(t, x, \xi) \leq \frac{\Lambda^q}{q} |P_0 \xi |^q, \qquad \forall (t,x,\xi) \in  \ms{U}\times\R^N ,
\ee
where $P_0$ is a given orthogonal projection matrix with $\Rank P_0 < N$.
\item \label{hyp:standing_f} $f$ is locally uniformly bounded below with positive part $f_+ \in L_{\rm{loc}}^p(\R\times\R^N)$, where $p > 1$ is sufficiently large; the precise condition is stated below in Equation \eqref{hyp:pqN}. 
\item \label{hyp:standing_A} $A\in\R^{N\times N}$ is a given constant matrix.
\item $A$ and $P_0$ satisfy the {\it Kalman rank condition} \cite{Kalman}: there exists $K \in \mb{N}$ such that
\be \label{hyp:Kalman}
\R^N = \Image(A^K P_0) + \Image(A^{K-1} P_0) + \ldots + \Image(P_0) 
\ee
(this is one of several equivalent forms, see e.g. \cite{Zabczyk}).
\end{enumerate}
\end{hyp}

Under these assumptions, we prove interior H\"older regularity for viscosity solutions of \eqref{eq:HJ_intro} (Corollary~\ref{cor:ViscSolutions}).

\begin{remark}[Comments on the Assumptions]

Although we motivated the form of \eqref{eq:HJ_intro} using the LTI control system \eqref{intro:control_sys0}, we emphasise that our results can be applied to HJ equations derived from more general control systems. This is because we require only upper and lower bounds on $H$ of the form \eqref{ineq:coerc}.
In particular we can consider control-affine systems of the form
\begin{equation}\label{intro:control_sys}
\dot\eta_{t} = A\eta_{t} + \sum_{i=1}^{m} \beta^i_t \zeta_i(t,\eta_t), \qquad \beta_t = (\beta_t^1, \ldots, \beta_t^m )\in \R^m ,
\end{equation}
with Lagrangian cost satisfying for some constants $0 < c \leq C$,
\be
c |\beta|^{q'} \leq L(t,x, \beta) \leq C |\beta|^{q'} \qquad \forall (t,x) \in \ms{U}, \; \forall \beta \in \R^m,
\ee
as long as the control functions $\zeta_i(t,x)$ are such that the matrix $Z : = \sum_{i=1}^{m} \zeta_i \otimes \zeta_i$ is bounded and uniformly elliptic on $P_0 \R^N$, i.e. for some $0 < \mu \leq M < +\infty$,
\be
\mu |P_0 \xi|^2 \leq \sum_{i=1}^m |\xi \cdot \zeta_i(t,x)|^2 \leq M |P_0 \xi|^2 \qquad \forall (t,x) \in \ms{U}, \; \forall \xi \in \R^N .
\ee
\end{remark}

While our analysis is inspired by the approaches taken in \cite{CanCar, Cardaliaguet-Silvestre}, new sets of ideas are necessary to overcome the nontrivial challenges posed by the lack of full coercivity of the Hamiltonian. The heart of our analysis is to rely on a {\it geometric framework} enforced by the setting of the underlying linear control system \eqref{intro:control_sys0}.

The connection between geometric structure, controllability properties and regularity has long been recognised in the study of (degenerate) parabolic PDEs. The celebrated work of H\"ormander \cite{Hormander} (see also \cite{FollandStein, RothschildStein}) identified that diffusion operators of the form $\partial_t - X_0 - \sum_{i=1}^m X_i^2$ ($X_i$ being smooth real vector fields on $\R^N$) are hypoelliptic if the Lie algebra generated by the vector fields $\partial_t - X_0$, $\{ X_i \}_{i=1}^m$ spans the full tangent space of $\R^{N+1}$. Meanwhile, similar bracket-generating conditions underly the Chow--Rashevskii theorem \cite{Chow, Rashevskii} in geometric control theory and sub-Riemannian geometry. The study of Hörmander operators and the related topic of systems of H\"ormander vector fields has launched a vast and rich field of research: we refer for example to the monographs \cite{BonLanUgu, Bramanti_book} for a modern overview of the topic. 

Among degenerate parabolic PDEs, the natural points of comparison for the HJ equations \eqref{eq:HJ_intro} studied in this article are the \emph{ultraparabolic equations of Kolmogorov type} \cite{Lanconelli-Polidoro, Polidoro94}, which are of the (here, divergence) form
 \be \label{eq:ultra}
 \partial_t u + \langle Ax, \nabla_x u \rangle - \div_x \left (P_0 \sigma(t, x) P_0\nabla_x u \right ) = 0 ,
 \ee
where $\sigma$ is a real symmetric matrix that is uniformly elliptic with respect to the subspace $P_0 \R^N$. These are prototypes of equations based on H\"ormander vector fields for which commutators with the drift vector field $X_0$ are essential for generating the full tangent space (i.e. unlike `sum of squares' heat operators of the form $\partial_t - \sum_{i=1}^m X_i^2$, in which the Lie algebra generated by $\{ X_i \}_{i=1}^m$ spans $\R^N$). Moreover, the class contains several specific examples of particular interest, including \emph{kinetic} equations arising in statistical physics. Applications to nonlinear models such as the Landau equation motivate the study of \eqref{eq:ultra} when $\sigma$ has low regularity (perhaps bounded measurable) and has prompted recent intensive study of regularity properties, Harnack inequalities, Poincar\'e-type inequalities and related questions for Kolmogorov-type equations \eqref{eq:ultra}, the kinetic case in particular, and nonlocal generalisations thereof: we refer to the reviews \cite{GolseReview, AncPiccReb, BriMou} for an account of these developments.
  The present article is in a somewhat similar spirit, and indeed part of our analysis is inspired by results from \cite{ADGLMR}. At the same time, the underlying regularising mechanism is rather different: it arises in \emph{first-order} PDEs and is tied inextricably to the nonlinearity.

\subsection*{Description of our main results.}

The Kalman rank condition \ref{hyp:Kalman} has several useful consequences. For example, any two points in $\R^N$ can be connected by a controlled trajectory in positive time, i.e. for any $x,y \in \R^N$ and $t > 0$, there exists a continuous control $\beta : [0,t] \to \R^N$ such that the corresponding solution $\eta$ of \eqref{intro:control_sys0} with $\eta(0) = x$ satisfies $\eta(t) = y$ \cite[Proposition 1.1, Theorem 1.2]{Zabczyk}. It also induces a useful decomposition of $\R^N$ that we will now introduce in order to be able to state our results.

\begin{defn}[ Orthogonal Decomposition] \label{def:OrthoDecomp}
Assuming the Kalman rank condition, we define $\kappa = \kappa(A, P_0)$ to be the smallest such $K$ for which \eqref{hyp:Kalman} holds.

Then $\R^N$ can be decomposed into orthogonal subspaces $(E_k : k = 0, \ldots, \kappa)$ in the following way. Let $V_0 = E_0 = \Image(P_0)$ and 
\be
V_k := \Image(A^k P_0) + \Image(A^{k-1} P_0) + \ldots + \Image(P_0) \qquad k = 1, \ldots \kappa .
\ee
Thus $A(V_k) \subset V_{k+1}$ for all $k = 1, \ldots, \kappa - 1$; recall that $V_\kappa = \R^N$ by \eqref{hyp:Kalman}.

Then let $(E_k : k = 1, \ldots, \kappa)$ be successive orthogonal complements such that
\be
V_k = V_{k-1} \oplus E_k , \qquad  k=1, \ldots, \kappa.
\ee
Hence $V_k  : = \bigoplus_{j=0}^k E_j$, and in particular $\R^N = \bigoplus_{k=0}^\kappa E_k$. 

The projection matrix onto $E_k$ is denoted by $P_k$.

Notice that, for each $k = 0, \ldots, \kappa$, 
\be \label{A_upper_triangular}
A(E_k) \subset  A(V_k) \subset V_{k+1} = \bigoplus_{j=0}^{k+1} E_j.
\ee
\end{defn}

\begin{defn}\label{def:omega}
Given $\alpha \in (0,1]$, the {\it anisotropic modulus of continuity} $\omega_\alpha$ is defined by
\be
\omega_\alpha(t,x) : =  |t|^\alpha + \sum_{j=0}^\kappa |P_j x|^{\frac{\alpha}{\alpha/q' + 1/q + j}} \qquad (t,x) \in \R \times \R^N .
\ee

\end{defn}

The main theorem of this paper can be formulated as follows.

\begin{thm} \label{thm:main}
Let $\ms{U} \subset \R \times \R^N$ be an open set. Let $q>1$, $c_0 \in  C(\ms{U})$ and $f_+ \in L^{p}_{\rm loc}(\ms{U})$ be given, where, with the previously introduced notion $n_{j}$ and $\kappa$,
\be\label{hyp:pqN} 
p >  N/q + 1 + \sum_{j=1}^\kappa j n_j . 
\ee 
Furthermore, let $u \in C_b(\ms{U})$ be a viscosity supersolution of
\be \label{eq:HJ_geq-thm}
\partial_t u + \left \langle Ax , \nabla_{x} u \right \rangle + \frac{\Lambda^q}{q} |P_0 \nabla_{x} u |^q + c_0 =0 \qquad \text{in} \; \ms{U}
\ee
and a viscosity subsolution of 
\be
\label{eq:HJ_leq-thm}
\partial_t u + \left \langle Ax , \nabla_{x} u \right \rangle + \frac{\lambda^q}{q} |P_0 \nabla_{x} u |^q - f  = 0 \qquad \text{in} \; \ms{U} .
\ee
Then $u$ is locally H\"older continuous in $\ms{U}$.

More precisely, there exists $\alpha \in (0,1)$ such that, for all compact subsets $K \subset \ms{U}$, there exists $C>0$ such that
\be \label{est:Holder_aniso}
|u(s,y) - u(t,x)| \leq C \omega_\alpha \left(t-s, y - e^{-(t-s)A} x\right) 
\ee
for all $(s,y), (t,x) \in K$ such that $s \leq t$.
The exponent $\alpha$ depends on $\| u \|_{L^\infty(\ms{U})}$, $\lambda$, $\Lambda$, $A$ and $q$. The constant $C>0$ additionally depends on a choice of open set $\ms{V}$ compactly contained in $\ms{U}$ with $K \subset \ms{V}$, on $\sup_{\ms{V}} c_0$ and on $\|f_+\|_{L^{p}(\ms{V})}$.
\end{thm}

\begin{remark}
The result extends to $u \in C(\ms{U})$ by localisation, since then $u \in C_b(K)$ for any compact $K \subset \ms{U}$. 
However, note that both the exponent $\alpha$ and the constant $C>0$ then depend on $K$.

Similarly, the coercivity constants $\lambda, \Lambda$ in \eqref{ineq:coerc} can be replaced by functions $\lambda, \Lambda : \ms{U} \to (0, + \infty)$ satisfying $\inf_{K} \lambda > 0$, $\sup_{K} \Lambda < + \infty$ for all compact sets $K \subset \ms{U}$. Both $\alpha$ and $C>0$ then depend on $K$.
\end{remark}

\begin{remark}
The anisotropic modulus of H\"older continuity arises from a scaling property of the equations \eqref{eq:HJ_geq-thm}-\eqref{eq:HJ_leq-thm} related to the geometric structure induced by $A$ and $P_0$. We discuss this in detail below in Section \ref{subsec:scaling}.

The appearance of the free flow $e^{-(t-s)A}$ inside the H\"older modulus in the estimate
\eqref{est:Holder_aniso} is natural in this context -- consider for example the case $\lambda = \Lambda = 0$, $c_0 = 0$, $f = 0$, which gives a pure transport equation whose solutions are constant along paths $t \mapsto e^{t A}x$. Similar estimates are found in the analogous setting for ultraparabolic equations (see e.g. \cite{DiFPol}).
Estimate \eqref{est:Holder_aniso} can be used to obtain a standard H\"older continuity estimate in terms of $|t-s|$ and $|y-x|$:
since
\be
| y - e^{-(t-s)A} x | \leq | y - x | + | (e^{-(t-s)A} - I) x | \leq |y-x| + C_A |t-s| |x| \qquad \text{for} \; |t-s| \leq 1,
\ee
there exists $C>0$ depending on $K$ such that
\be \label{est:Holder_noflow}
|u(s,y) - u(t,x)| \leq C \left ( |t-s|^{\alpha \min \left\{1 , \frac{1}{\alpha/q' + 1/q + \kappa} \right\}} + \sum_{j=0}^\kappa |P_j (y-x)|^{\frac{\alpha}{\alpha/q' + 1/q + j}} \right ) .
\ee
However, the form \eqref{est:Holder_aniso} gives a finer description of the anisotropic H\"older continuity and its relationship with the flow induced by $A$.
\end{remark}

\begin{cor} \label{cor:ViscSolutions}
Let $\ms{U} \subset \R \times \R^N$ be an open set, and suppose that $u \in C_b(\ms{U})$ is a viscosity solution of 
\begin{align} \label{eq:HJ_main_cor}
\partial_t u + \left \langle Ax , \nabla_x u \right \rangle + H(t, x, \nabla_x u)  = f \qquad \text{in} \; \ms{U} ,
\end{align}
where $H$, $f$, $A$ and $P_0$ satisfy Assumption~\ref{hyp:standing}. Then $u$ is locally H\"older continuous in $\ms{U}$.
\end{cor}

\subsection*{The strategy of the proof of the main theorem.}

As this is typically done for H\"older type regularity theory, e.g. \`a la De Giorgi--Nash--Moser and related approaches, to prove Theorem \ref{thm:main} we proceed in two steps. Both of these have to be carefully tailored to our geometric setting, and we now summarise the high level guiding principles and ideas behind them.

\medskip

{\it Step 1. Rescaling.} The heart of our analysis is based on carefully designed rescaling operators, which enable us to move from large to small scales while preserving the principal terms of the HJ equation. Recalling the decomposition of the state space $\R^{N}$ into the subspaces $(E_{i})_{i=0,\dots,\kappa}$, using the matrices $A,P_{0}$ from the control system as described above, we observe that information propagates in an {\it anisotropic way} in each of these subspaces. In order to capture this, each of the $E_{i}$ spaces must scale differently. 
Correspondingly, in Section \ref{sec:geometry} we introduce a family of scaling transformations, that are related to a Lie group structure previously identified in the context of ultraparabolic equations \cite{Lanconelli-Polidoro}; however here the scaling must be tailored to the HJ structure, and in particular to the coercivity exponent $q$. The transformations lead to the construction of natural time-space cylinder-like domains, upon which the localised analysis is performed. The rescaling of the HJ equation does not preserve the drift matrix $A$: rather, at smaller and smaller scales, interestingly the so-called {\it principal part of $A$} is essentially responsible for the drift. The presence of this anisotropic phenomenon leads to the definition of the H\"older modulus of continuity $\omega_{\alpha}$ in Definition \ref{def:omega}, with the help of which we measure the regularity of the solutions. 

\medskip

{\it Step 2. Improvement of oscillations from larger to smaller scales.} Armed with the natural geometric setup designed in {\it Step 1}, in Proposition \ref{prop:ImproveOsc} we prove the crucial \emph{improvement of oscillation} property, which eventually leads to the desired H\"older estimates. This is achieved through comparison with specific families of sub- and supersolutions of the HJ equation with carefully designed boundary data, constructed through an optimal control representation.
The argument is carried out in two sub-steps: 

First, we show the improvement of {\it upper bounds}. 
To show the desired upper bound on our supersolutions, we construct particular admissible trajectories, which must lie within the given cylinder-like domains. Through suitable bounds on the Lagrangian cost associated to these trajectories, we deduce the necessary upper bounds if the source term $f$ is of class $L^{\infty}$. For source terms of class $L^{p}$ much more care is needed, as estimates that rely on the time-space $L^{p}$ average of $f$ {are required}. For this, the introduction of special conic type neighbourhoods around the previously constructed test trajectories is necessary. These neighbourhoods both have positive Lebesgue measure in $\R \times \R^N$, and are built out of solutions of the control system, and therefore allow us to obtain bounds on supersolutions in terms of $L^p$ bounds on the source, provided that the corresponding coordinate transformation is well understood with suitable properties. In particular, for this approach to be successful, the trajectories must be \emph{curved}. Similar ideas have been used e.g. in \cite{CardaliaguetGraber} in regularity estimates for HJ equations in the non-degenerate case without drift, however the construction of suitable trajectories becomes considerably more involved here since they are also required to obey the control system \eqref{intro:control_sys0}. Our approach here is inspired by similar constructions in \cite{ADGLMR} for the ultraparabolic case.

Second, to obtain the necessary {\it lower bound} improvements, we work with specific subsolutions. However for the lower bounds there is no choice but to work with {\it optimal trajectories}. The analysis then becomes a case of understanding where these optimal trajectories might hit the boundary of the cylinder like domains, and their control cost. 

\medskip

The final result on the H\"older estimates is then proven by an iterative argument, alternating the two previous steps; this is the subject of Proposition \ref{prop:small_Holder}, when the source term $f$ is suitably `small' in $L^{p}$-norm. For the general case, one final rescaling argument is needed, performed in Subsection \ref{subsec:general}, which leads to the last step in the proof of Theorem \ref{thm:main}.

\subsection*{Literature review relevant to our results.}

As mentioned above, our results can seen as a natural continuation of the line of works \cite{Cardaliaguet2009, CardaliaguetGraber, CarRai, CanCar, Cardaliaguet-Silvestre, ChanVasseur, CarSee, ChanVasseur, StoVas, Cir, Barles}, studying H\"older estimates of solutions to HJ equations. In all these references, however, the uniformly non-degenerate coercivity of the underlying Hamiltonians played a crucial role. Other significant regularity estimates (such as Sobolev, Lipschitz, $L^q$, etc.) for HJ type equations were obtained in \cite{CDLeoniPor, CarPorTon, CirantGoffi_Lipschitz, CirantGoffi_Maximal}.
One motivation for studying H\"older regularity of solutions to HJ equations comes from problems related to homogenisation, as in \cite{Schwab, JingSouTran, SM_homog}. 

When it comes to PDEs in various geometric settings, quantified regularity results for solutions to HJ equations are sparse. The work \cite{BarFelSor} seems to be the only work studying the H\"older regularity question for solutions to a particular Eikonal type equations with Lie brackets. Other properties of solutions to HJ equations associated to geometric control problems (semi-concavity and Lipschitz regularity for value functions, Hopf-Lax-type representation formulae) were investigated in \cite{BalCalPin, DraLiuZhang}, as well as in \cite{DragoniFeleqi, ManMarTchou_Heisenberg} as part of initial analysis required for the study of mean field games. Similarly, in applications to mean field games and related models, general control systems and properties of the resulting HJB equations are of interest. For a non-exhaustive list of works in this direction we refer to \cite{CutManMarTchou, AchManMarTchou, CanMen, MimStam,  BarCar_LargeDiscount, GPM1, AGPM, Achdou, GianattiSilva}. Our hope is that the regularity results obtained in the current paper will be beneficial for the study of some fine regularity properties of solutions to related mean field games systems.

Fine regularity properties of solutions to other kinds of PDEs of ultraparabolic and kinetic type have received a huge attention: for a very brief selection see for example \cite{Lanconelli-Polidoro, WangZhang_ultra, WangZhang_ultra2, GIMV, ImbertSilvestre}, the reviews \cite{GolseReview, AncPiccReb, BriMou} and the references therein. The works that have particularly inspired part of our analysis are \cite{Lanconelli-Polidoro, ADGLMR}.

\medskip

The structure of the rest of the paper is as follows. Section \ref{sec:properties} contains all the necessary preliminary analysis, detailing various properties of the control system, the geometric framework, the rescaling operations and their action on solutions to the HJ equation, and some further properties of HJ equations. Section \ref{sec:holder} is a main part of the paper containing the elements of the proof of our main theorems. Here we distinguish the two main cases, first presenting the case when the source term is small, and then the general case.

\subsection*{Acknowledgements.}

ARM has been partially supported by the EPSRC New Investigator Award ``Mean Field Games and Master equations'' under award no. EP/X020320/1.

\section{Properties of the Control System and the HJ Equation}\label{sec:properties}
This section collects a few important properties of the controlled ODE
\be \label{eq:Properties_ODE}
\dot \eta = A \eta + P_0 \beta \quad {\text{on an interval}}\ (T_1,T_2),
\ee
which will be used for the rest of the paper. Here $\eta:(T_1,T_2)\to\R^N$ stands for the state variable, while $\beta:(T_1,T_2)\to\R^N$ is the control variable. Precise assumptions on $\beta$ will be made later.

\subsection{Geometric setup}\label{sec:geometry}

The geometric structure we outline here is based on the corresponding analysis for ultraparabolic equations developed in \cite{Lanconelli-Polidoro} (see also the survey \cite{Anceschi-Polidoro}).

\subsubsection{Choice of Basis and the Principal Part Operator}

Under the Kalman rank condition \eqref{hyp:Kalman}, there exists a basis for $\R^N$ such that vectors $x \in \R^N$ are represented in the form $x =(P_\kappa x, \ldots, P_0 x)^\top\in\R^N$, where $\{ P_j \}_{j=0}^\kappa$ denote the projection matrices defined in Definition~\ref{def:OrthoDecomp}. By \eqref{A_upper_triangular}, in this basis the drift matrix $A$ is then of the form
(see \cite{Lanconelli-Polidoro, Anceschi-Polidoro})
\be \label{A_BlockForm}
A = \begin{pmatrix} 
A^{(\kappa, \kappa)} & A^{(\kappa, \kappa-1)} & \mb{O}_{n_\kappa \times n_{\kappa-2}}  & \mb{O}_{n_\kappa \times n_{\kappa-3}}  & \cdots & \mb{O}_{n_\kappa \times n_{0}} \\
A^{(\kappa -1 , \kappa)} & A^{(\kappa-1 , \kappa-1)} & A^{(\kappa-1, \kappa-2)} & \mb{O}_{n_{\kappa-1} \times n_{k-3}}  & \cdots & \mb{O}_{n_{\kappa-1} \times n_{0}} \\
&&&&&\\
\vdots & \vdots & & \ddots & \ddots  & \vdots \\
&&&&&\\
A^{(2 , \kappa)} & A^{(2, \kappa-1)} & \cdots &  & A^{(2,1)}  & \mb{O}_{n_{2} \times n_{0}} \\
A^{(1 , \kappa)} & A^{(1, \kappa-1)}  & \cdots  &  & A^{(1,1)} & A^{(1,0)}\\
A^{(0 , \kappa)} & A^{(0, \kappa-1)}  & \cdots &  &  & A^{(0,0)}
\end{pmatrix}
\ee
where: 
\begin{enumerate}[(i)]
\item $n_j : = \dim E_j$ for $j = 0, \ldots, \kappa$. Notice that $\ds \sum_{j=0}^\kappa n_j = N$.
\item Each block $A^{(j,j-1)} \in \R^{n_j \times n_{j-1}}$ ($j = 1, \ldots, \kappa$) on the {\it upper off-diagonal} is a $n_j \times n_{j-1}$ matrix of (\textbf{full}) rank $n_j$.
\item Each block $A^{(i,j)} \in \R^{n_i \times n_j}$ ($i \leq j$) is a $n_i \times n_{j}$ matrix.
\item Each null block $\mb{O}_{n_{i} \times n_{j}} \in \R^{n_i \times n_j}$ ($i>j+1$) is a $n_i \times n_j$ zero matrix.
\end{enumerate}

The upper off-diagonal blocks are key to understanding the behaviour of $A$ in the present context, and it is helpful to consider the corresponding \emph{principal part} (see \cite{Anceschi-Polidoro} and the references therein).

\begin{defn} \label{def:PrincipalPart}
The \emph{principal part} of $A$ is the matrix $A_0$ given by
\be
A_0 : = \sum_{j=0}^{\kappa - 1} P_{j+1} A P_j,
\ee
or equivalently, in block form as in \eqref{A_BlockForm}, by
\be \label{A0_BlockForm}
A_0 = \begin{pmatrix} 
\mb{O}_{n_\kappa \times n_{\kappa}}  & A^{(\kappa, \kappa-1)} & \mb{O}_{n_\kappa \times n_{\kappa-2}}  & \mb{O}_{n_\kappa \times n_{\kappa-3}}  & \cdots & \mb{O}_{n_\kappa \times n_{0}} \\
\mb{O}_{n_{\kappa-1} \times n_{\kappa}}  & \mb{O}_{n_{\kappa-1} \times n_{\kappa-1}}  & A^{(\kappa-1, \kappa-2)} & \mb{O}_{n_{\kappa-1} \times n_{k-3}}  & \cdots & \mb{O}_{n_{\kappa-1} \times n_{0}} \\
&&&&&\\
\vdots & \vdots & & \ddots & \ddots  & \vdots \\
&&&&&\\
\mb{O}_{n_{2} \times n_{\kappa}}  & \mb{O}_{n_{2} \times n_{\kappa-1}} & \cdots &  & A^{(2,1)}  & \mb{O}_{n_{2} \times n_{0}} \\
\mb{O}_{n_{1} \times n_{\kappa}}  & \mb{O}_{n_{1} \times n_{\kappa-1}}   & \cdots  &  & \mb{O}_{n_{1} \times n_{1}} & A^{(1,0)}\\
\mb{O}_{n_{0} \times n_{\kappa}}  & \mb{O}_{n_{0} \times n_{\kappa-1}}   & \cdots &  &  & \mb{O}_{n_{0} \times n_{0}} 
\end{pmatrix} .
\ee

Note that $A_0$ is a nilpotent matrix with $A_0^{\kappa+1} = \mb{O}_{N \times N}$ and $\ker(A_0) = E_\kappa$.
\end{defn}

\subsubsection{Scaling Transformations}\label{subsec:scaling}

One interpretation of the principal part is that it arises from a certain rescaling of the state space $\R^N$. This is based on the following family of transformations.

\begin{defn} \label{def:S}
\begin{enumerate}[(i)]
\item  The scaling operator $\R \ni r\mapsto S(r) \in \R^{N\times N}$ is defined by
\be\label{eq:def_S}
S(r) = \sum_{i=0}^\kappa r^{i} P_i .
\ee
$S(r)$ is invertible for all $r \neq 0$ with $S(r)^{-1} = S(r^{-1})$. In the case $r=0$, $S(0) = P_0$, which is not invertible for $n_0 < N$.

$S(r)$ has operator norm
\be \label{eq:Sr_norm}
\| S(r) \| = \max_{i=0, \ldots, \kappa} |r|^i = \max \{ 1, |r|^\kappa \} .
\ee

\item For $\gamma , r > 0$, the anisotropic \emph{spatial} dilation $D_r = D_r^{(\gamma)}\in \R^{N\times N}$ is defined by $D_r : = r^\gamma S(r)$. Note that $D_r$ is invertible with $D_r^{-1} = D_{1/r} = r^{-\gamma} S(r^{-1})$.

\item For $\gamma, r > 0$ the anisotropic \emph{space-time} dilation $\widetilde D_r = \widetilde D_r^{(\gamma)}\in\R^{(N+1)\times(N+1)}$ is defined by
\be
\widetilde D_r^{(\gamma)} \begin{pmatrix}
t \\ x
\end{pmatrix} := \begin{pmatrix}
r t \\ D_r^{(\gamma)} x
\end{pmatrix}.
\ee
Note that $\widetilde D_r^{(\gamma)}$ is invertible with $(\widetilde D_r^{(\gamma)})^{-1} = \widetilde D_{1/r}^{(\gamma)}$.

Moreover,
\be \label{eq:detDrTilde}
\det \widetilde D_r^{(\gamma)} =r \det (D_r^{(\gamma)}) = r^{N\gamma+1}\det S(r) = r^{N \gamma + 1 + \sum_{j=1}^\kappa j n_j  }.
\ee 
\end{enumerate}
\end{defn}

\begin{paragraph}{Notation.}
We represent elements of $\R \times \R^N$ interchangeably either as column vectors $\begin{pmatrix}
t \\ x
\end{pmatrix}$ or as comma separated pairs $(t,x)$. 

We use the same notation (e.g. $S(r)$) for a matrix and the linear map it represents. For these matrices the norm $\| \cdot \|$ by default refers to the operator norm, e.g.
\be
\| S(r) \| : =  \sup_{x \in \R^N, \, |x| = 1} \| S(r) x \| .
\ee
\end{paragraph}
 
This definition of $\widetilde D_r$ is partly inspired by \cite{Lanconelli-Polidoro}, however our specific choice here is motivated by its effect when used to rescale the Hamilton--Jacobi equations.
We will now describe this procedure: broadly speaking, our aim is to be able to make the zero order terms negligible compared to the first order terms at small scales.

\begin{lemma} \label{lem:HJ_rescaled}
Let $\ms{U} \subset \R \times \R^N$ be an open set.
Let $u \in C(\ms{U})$ be a viscosity supersolution of
\be \label{eq:HJ_geq-1}
\partial_t u + \left \langle Ax , \nabla_{x} u \right \rangle + \frac{\Lambda^q}{q} |P_0 \nabla_{x} u |^q + c_0 = 0  \qquad  \text{in} \; \ms{U}
\ee
and a viscosity subsolution of
\be \label{eq:HJ_leq-1}
\partial_t u + \left \langle Ax , \nabla_{x} u \right \rangle + \frac{\lambda^q}{q} |P_0 \nabla_{x} u |^q - f = 0 , \qquad \text{in} \; \ms{U},
\ee
where $c_{0}\in\R$ and $f \in L^p(\ms{U})$ is continuous.

Let $0 \leq \alpha \leq 1$ and
define for $r > 0$ the rescaling $u_r \in C(\widetilde D_{1/r}^{(\gamma)} \ms{U})$  by
\be
u_r : = r^{- \alpha} u \circ \widetilde D_r^{(\gamma)} , \qquad 
\text{with} \; \gamma = \frac{1}{q} + \frac{\alpha}{q'} \in \left [ \frac{1}{q}, 1 \right ].
\ee
Then $u_r$  is a viscosity supersolution of
\be \label{eq:HJ_ur_geq}
\partial_t u_r + \left \langle r S(r)^{-1} A S(r) x ,  \nabla_x u_r \right \rangle + \frac{\Lambda^q}{q} | P_0 \nabla_x u_r |^q + r^{1-\alpha } c_0 = 0 \qquad  \text{in} \; \widetilde D_{1/r}^{(\gamma)} \ms{U}
\ee
and a viscosity subsolution of
\be \label{eq:HJ_ur_leq}
\partial_t u_r + \left \langle r S(r)^{-1} A S(r) x , \nabla_x u_r \right \rangle + \frac{\lambda^q}{q} |P_0 \nabla_x u_r |^q - r^{1-\alpha} f  \circ \widetilde D_{r}^{(\gamma)} = 0 \qquad  \text{in} \; \widetilde D_{1/r}^{(\gamma)} \ms{U} .
\ee
The rescaled source term $r^{1-\alpha} f  \circ \widetilde D_{r}^{(\gamma)}$ has $L^p$ norm
\be
\| r^{1-\alpha} f  \circ \widetilde D_{r}^{(\gamma)} \|_{L^p(\widetilde D_{1/r}^{(\gamma)}\ms{U})} = r^{ 1 - \frac{1}{p}  (\frac{N}{ q} + 1 + \sum_{i=1}^\kappa i n_i ) - \alpha \left (1 + \frac{N}{ p q '} \right ) } \| f \|_{L^p(\ms{U})} .
\ee

\end{lemma}

\begin{proof}

For brevity, we perform the rescaling under the assumption that $u$ is differentiable with the inequalities 
\begin{align} 
\partial_t u + \left \langle Ax , \nabla_{x} u \right \rangle + \frac{\Lambda^q}{q} |P_0 \nabla_{x} u |^q + c_0 \geq 0 \\ 
\partial_t u + \left \langle Ax , \nabla_{x} u \right \rangle + \frac{\lambda^q}{q} |P_0 \nabla_{x} u |^q - f \leq 0 ,
\end{align}
satisfied pointwise on $\ms{U}$. The result for viscosity sub/supersolutions is proved similarly, by transferring the manipulations to the test function.

By direct computation, $u_r$ satisfies
\be 
\partial_t u_r = r^{1 - \alpha} \partial_t u \circ \widetilde D_r^{(\gamma)}, \qquad \nabla_x u_r = r^{- \alpha} D_r^{(\gamma)} ( \nabla_x u \circ \widetilde D_r^{(\gamma)}) = r^{\gamma - \alpha} S(r) ( \nabla_x u \circ \widetilde D_r^{(\gamma)})
\ee
We substitute these expressions to find that, for example,
\be
r^{\alpha - 1} \partial_t u_r \circ \widetilde D_{1/r}^{(\gamma)} +  r^{\alpha - \gamma} \left \langle Ax , S (r^{-1}) ( \nabla_{x} u_r \circ \widetilde D_{1/r}^{(\gamma)} ) \right \rangle + r^{q(\alpha - \gamma)} \frac{\Lambda^q}{q} |P_0 S (r^{-1}) ( \nabla_{x} u_r  \circ \widetilde D_{1/r}^{(\gamma)}) |^q + c_0 \geq 0  \quad \text{in} \; \ms{U} .
\ee
For the drift term, we must pay attention to the $x$ dependence: observe that
\be 
\left \langle Ax , S (r^{-1}) ( \nabla_{x} u_r \circ \widetilde D_{1/r}^{(\gamma)} ) \right \rangle = \left \langle A  D_{r}^{(\gamma)}x , S (r^{-1}) \nabla_{x} u_r \right \rangle \circ \widetilde D_{1/r}^{(\gamma)} = r^\gamma \left \langle S (r^{-1}) A S(r) x ,  \nabla_{x} u_r \right \rangle \circ \widetilde D_{1/r}^{(\gamma)}
\ee
where in the second inequality we have used that $S(r^{-1})$ is symmetric and $D_{r}^{(\gamma)} = r^\gamma S(r)$.
Noting also that $P_0 S (r^{-1}) = P_0$, and recalling that $c_0$ is constant, we deduce that
\be
r^{\alpha - 1}\left [ \partial_t u_r +  r \left \langle S (r^{-1}) A S(r) x ,  \nabla_{x} u_r \right \rangle + r^{1 - \alpha + q(\alpha - \gamma)} \frac{\Lambda^q}{q} |P_0 \nabla_{x} u_r |^q + r^{1-\alpha} c_0  \circ \widetilde D_{r}^{(\gamma)} \right ] \circ \widetilde D_{1/r}^{(\gamma)} \geq 0  \quad \text{in} \; \ms{U}
\ee
Since $\alpha - \gamma = \frac{\alpha - 1}{q}$, we obtain
\be
 \partial_t u_r +  \left \langle r S (r^{-1}) A S(r) x ,  \nabla_{x} u_r \right \rangle + \frac{\Lambda^q}{q} |P_0 \nabla_{x} u_r |^q + r^{1-\alpha} c_0  \circ \widetilde D_{r}^{(\gamma)} \geq 0  \qquad \text{in} \; \widetilde D_{1/r}^{(\gamma)} \ms{U},
\ee
and similarly
\be
 \partial_t u_r +  \left \langle r S (r^{-1}) A S(r) x ,  \nabla_{x} u_r \right \rangle + \frac{\lambda^q}{q} |P_0 \nabla_{x} u_r |^q + r^{1-\alpha} f  \circ \widetilde D_{r}^{(\gamma)} \leq 0  \qquad \text{in} \; \widetilde D_{1/r}^{(\gamma)} \ms{U} .
\ee

We compute $\| f \circ \widetilde D_r^{(\gamma)} \|_{L^p}$ using the formula \eqref{eq:detDrTilde} for $\det \widetilde D_r^{(\gamma)}$:
\begin{align}
r^{1-\alpha }  \| f \circ \widetilde D_r^{(\gamma)} \|_{L^p(\widetilde D_{1/r}^{(\gamma)} \mc{U})} & = r^{1-\alpha } |\det \widetilde D_r^{(\gamma)}|^{-1/p} \| f \|_{L^p(\ms{U})} \\
& =   r^{1-\alpha - \frac{1}{p}(N \gamma + (1 + \sum_{i=1}^\kappa i n_i ))} \| f \|_{L^p(\ms{U})} .
\end{align}
Since $\gamma = \frac{1}{q} + \frac{\alpha}{q'}$, we conclude that
\begin{align}
r^{1-\alpha }  \| f \circ \widetilde D_r^{(\gamma)} \|_{L^p(\mc{U})} & = r^{ 1 - \frac{1}{p}  (\frac{N}{ q} + 1 + \sum_{i=1}^\kappa i n_i ) - \alpha \left (1 + \frac{N}{ p q '} \right ) } \| f \|_{L^p(\widetilde D_r^{(\gamma)} \mc{U})}. \label{eq:f_PRescalingCost}
\end{align}
\end{proof}

We observe that the rescaling preserves the time derivative and nonlinear terms, but not the drift term $\langle Ax, \nabla_x u \rangle$. However, we will see that the rescaled drift matrices $r S(r^{-1}) A S(r)$ converge, as $r \to 0$, to the principal part of $A$. For convenience we introduce the following notation for these matrices.

\begin{defn}
For $h \geq 0$, $A_h$ denotes the 
anisotropic rescaling of $A$ defined by
\be
A_h : = \sum_{j=0}^\kappa \sum_{i =0}^{\kappa \wedge (j+1)} h^{j + 1 - i} P_i A P_j ,
\ee 
or, in block form as in \eqref{A_BlockForm} by
\be \label{Ar_BlockForm}
A_h = \begin{pmatrix} 
h A^{(\kappa, \kappa)} & A^{(\kappa, \kappa-1)} & \mb{O}_{n_\kappa \times n_{\kappa-2}}  & \mb{O}_{n_\kappa \times n_{\kappa-3}}  & \cdots & \mb{O}_{n_\kappa \times n_{0}} \\
h^2 A^{(\kappa -1 , \kappa)} & h A^{(\kappa-1 , \kappa-1)} & A^{(\kappa-1, \kappa-2)} & \mb{O}_{n_{\kappa-1} \times n_{k-3}}  & \cdots & \mb{O}_{n_{\kappa-1} \times n_{0}} \\
&&&&&\\
\vdots & \vdots & & \ddots & \ddots  & \vdots \\
&&&&&\\
h^{\kappa-1} A^{(2 , \kappa)} & h^{\kappa-2} A^{(2, \kappa-1)} & \cdots &  & A^{(2,1)}  & \mb{O}_{n_{2} \times n_{0}} \\
h^{\kappa }  A^{(1 , \kappa)} & h^{\kappa-1} A^{(1, \kappa-1)}  & \cdots  &  & h A^{(1,1)} & A^{(1,0)}\\
h^{\kappa+1}  A^{(0 , \kappa)} & h^{\kappa } A^{(0, \kappa-1)}  & \cdots &  &  & h A^{(0,0)}
\end{pmatrix} .
\ee
Note that the case $h=0$ is consistent with Definition~\ref{def:PrincipalPart}.  Furthermore, for $h>0$, $A_h$ may equivalently be defined by
\be
A_h =  h S(h^{-1}) A S(h)  =  \sum_{j=0}^\kappa \sum_{i =0}^{\kappa \wedge (j+1)} h^{j + 1 - i} P_i A P_j .
\ee 

\end{defn}

\begin{remark}
Observe that:
\begin{enumerate}[(i)]
\item For any $h > 0$, $A_0$ is invariant under the rescaling
\be \label{eq:ScalingPreservesA0}
h S(h^{-1}) A_0 S(h) = \sum_{j=0}^\kappa \sum_{i =0}^{\kappa \wedge (j+1)} h^{j + 1 - i} P_i A_0 P_j = \sum_{j=0}^\kappa \sum_{i =0}^{\kappa \wedge (j+1)} \sum_{l=0}^{\kappa - 1} h^{j + 1 - i} P_i P_{l+1} A P_l P_j = \sum_{l=0}^{\kappa - 1} P_{l+1} A P_l = A_0 .
\ee
\item $\lim_{h \to 0} A_h = A_0$ with 
\be
A_h - A_0 = h \sum_{j=0}^\kappa \sum_{i=0}^j h^{j - i} P_i A P_j,
\ee
such that there exists a constant $C>0$ (depending on the choice of the matrix norm $\| \cdot \|$) for which
$\| A_h - A_0 \| \leq C h$ for all $0 < h \leq 1$.
\end{enumerate}
\end{remark}

Lemma~\ref{lem:HJ_rescaled} can then be rewritten in the following form.

\begin{cor} \label{rmk:HJ_rescaled_Ar}
Let $u$, $\alpha$ and $\gamma$ satisfy the hypotheses of Lemma~\ref{lem:HJ_rescaled}. Then, for $r>0$, $u_r$ is a viscosity
supersolution of
\be \label{eq:HJ_ur_geq}
\partial_t u_r + \left \langle A_r x ,  \nabla_x u_r \right \rangle + \frac{\Lambda^q}{q} | P_0 \nabla_x u_r |^q + c^{[r]} = 0 \qquad  \text{in} \; \widetilde D_{1/r}^{(\gamma)} \ms{U}
\ee
and a viscosity subsolution of
\be \label{eq:HJ_ur_leq}
\partial_t u_r + \left \langle A_r x , \nabla_x u_r \right \rangle + \frac{\lambda^q}{q} |P_0 \nabla_x u_r |^q - f^{[r]} = 0 \qquad  \text{in} \; \widetilde D_{1/r}^{(\gamma)} \ms{U} ,
\ee
where
\be
c^{[r]} : = r^{1- \alpha} c_0, \qquad f^{[r]} : =  r^{1-\alpha} f  \circ \widetilde D_{r}^{(\gamma)},
\ee
such that
\be \label{eq:fr_Lp}
\| f^{[r]} \|_{L^p(\widetilde D_{1/r}^{(\gamma)}\ms{U})} =  r^{ 1 - \frac{1}{p}  (\frac{N}{ q} + 1 + \sum_{i=1}^\kappa i n_i ) - \alpha \left (1 + \frac{N}{ p q '} \right ) } \| f \|_{L^p(\ms{U})} .
\ee
\end{cor}

\begin{remark}

The exponent of $r$ in Equation \eqref{eq:fr_Lp} is strictly positive for 
\be \label{alpha_range}
\alpha < \frac{p - ( \frac{N}{q} + 1 + \sum_{j=1}^\kappa j n_j) }{ p + \frac{N}{ q '}} .
\ee
Such an $\alpha \geq 0$ exists for any $p$ in the range
\be \label{p_range}
p  >  N/q + 1 + \sum_{j=1}^\kappa j n_j .
\ee
For $\alpha \in [0,1)$ satisfying \eqref{alpha_range} it is therefore possible to make the source terms in \eqref{eq:HJ_ur_geq}-\eqref{eq:HJ_ur_leq} as small as desired by taking $r$ sufficiently small.
\end{remark}

\subsubsection{Cylinders}

\begin{defn} \label{def:Omega}

We define a reference domain in the $x$ variable by
\be
\Omega_1 : = \{ x \in \R^N : |x| < 1 \}.
\ee
Then, for any $r > 0$ and $\gamma \in [0,1]$, let
\be \label{def:Omega_r}
\Omega_r^\gamma : = D_r^{(\gamma)} \Omega_1 =  \{ x \in \R^N : |S(r)^{-1} x| < r^\gamma \}.
\ee
\end{defn}

We then define space-time cylinders. These are chosen to follow the flow induced by one of the matrices $A_h$.
\begin{defn}\label{def:Q_r}
Let $r, h > 0$ and $\gamma \in [0,1]$. The cylinder $Q^{h, \gamma}_r$ is defined by
\be
Q_r^{h, \gamma} : = \left \{ (t,x): t \in [- r , 0], \; e^{-t A_{h}} x \in \Omega_r^\gamma \right \} ,
\ee
where $\Omega_r^\gamma$ is the domain defined in equation \eqref{def:Omega_r}.
\end{defn}

\begin{remark} \label{rmk:CylinderTransformation}

We note the following properties of the cylinder $Q^{h, \gamma}_r$ with respect to the transformations in Definition~\ref{def:S}.

\begin{enumerate}[(i)]
\item For $\rho, h >0$, since $S(\rho) A_h S(\rho)^{-1} = h S\left(\frac{h}{\rho}\right)^{-1} A S\left(\frac{h}{\rho}\right) = \rho A_{h \rho^{-1}}$,
\begin{align}
\widetilde D_\rho^\gamma Q^{h, \gamma}_r & = \left \{ (t,x) \; : \; - \rho r \leq t \leq 0, \quad \rho^{- \gamma} e^{- \frac{t}{\rho} A_{h}} S(\rho)^{-1} x \in \Omega_r^\gamma \right \} \\
& = \left \{ (t,x) \; : \; - \rho r \leq t \leq 0, \quad \left | S(r)^{-1} e^{- \frac{t}{\rho} A_{h}} S(\rho)^{-1} x \right | < (\rho r)^\gamma \right \} \\
& = \left \{ (t,x) \; : \; - \rho r \leq t \leq 0, \quad \left | S(\rho r )^{-1} e^{-t  A_{h \rho^{-1}}} x \right | < (\rho r)^\gamma \right \} \\
& = Q^{h\rho^{-1}, \gamma}_{\rho r} \label{eq:Cylinder_dil}
\end{align}
\item In the case $h=0$, by \eqref{eq:ScalingPreservesA0} for any $\rho>0$ we have $S(\rho) A_0 S(\rho)^{-1} = \rho  A_0$. Hence
\be
 \widetilde D_\rho^\gamma Q^{0, \gamma}_r = Q^{0, \gamma}_{\rho r} .
\ee
\item Since $\Omega_1^\gamma$ is independent of $\gamma$, so is $Q^{h, \gamma}_1$. We therefore use the abbreviated notation $Q^{h}_1$.
\end{enumerate}

\end{remark}

\subsubsection{Group Structure}

The free flow $\dot \eta = A_h \eta$ gives rise to an associated family (in fact a Lie group) of transformations of $\R \times \R^{N}$ indexed by $\R \times \R^{N}$. These transformations are compatible with the flow and will therefore be useful for us as the natural translations for moving around $\R \times \R^{N}$. For a fuller discussion of this Lie group structure, particularly in the context of Kolmogorov-type operators, see \cite{Anceschi-Polidoro} and the references therein.

For each $h \geq 0$, the binary operation $\diamond_h$ defined by
\be \label{def:LG_Op}
(\tau , \zeta) \diamond_h (t,x)  : = (\tau + t , x + e^{t A_h} \zeta ) \qquad \forall \tau ,t \in \R, \; \zeta, x \in \R^N 
\ee 
gives $\R \times \R^{N}$ a Lie group structure.
The identity element of the group is $(0,0)$ and inverses are given by
\be
(t,x)^{-1}_h : = (-t , - e^{-t A_h}x) .
\ee
Each $(\tau, \zeta) \in \R^{N+1}$ acts on $\R \times \R^{N}$ via the associated {\it left translation} map $l_{(\tau, \zeta)}^h :  \R \times \R^{N} \to \R \times \R^{N}$, where
\be \label{def:LT}
 l_{(\tau, \zeta)}^h (t, x) : = (\tau , \zeta) \diamond_h (t,x) .
\ee

The relevance of this structure to the present setting is that the left translations preserve the Hamilton--Jacobi equation: for all differentiable $u$, 
\be
\partial_t (u \circ l_{(\tau, \zeta)}^h) + x^\top A^\top_h \nabla_x (u \circ l_{(\tau, \zeta)}^h) + \frac{\Lambda^q}{q} |P_0 \nabla_x (u \circ l_{(\tau, \zeta)}^h) |^q = \left(\partial_t u + x^\top A^\top_h \nabla_x u + \frac{\Lambda^q}{q} |P_0 \nabla_x u|^q \right) \circ l_{(\tau, \zeta)}^h .
\ee

The group structure interacts with the space-time dilations $\widetilde D_r^\gamma$ in the following way.

\begin{lemma}
Let $\gamma, h \geq 0$. Then, for any $r>0$, the following hold:
\begin{enumerate}[(i)]
\item For all $(\tau , \zeta), (t,x) \in \R \times \R^N$,
\be \label{eq:LG_dil}
\widetilde D_r^\gamma(\tau , \zeta) \diamond_h \widetilde D_r^\gamma (t,x) = \widetilde D_r^\gamma \left[ (\tau , \zeta) \diamond_{hr} (t,x) \right] .
\ee
\item For all $(\tau , \zeta) \in \R \times \R^N$, the left translation satisfies
\be \label{eq:LT_dil}
 l_{\widetilde D_r^\gamma (\tau, \zeta)}^h \circ \widetilde D_r^\gamma = \widetilde D_r^\gamma \circ  l_{(\tau, \zeta)}^{hr } .
\ee
\item The inverse map satisfies
\be \label{eq:Inverse_dil}
\left ( \widetilde D_{r}^\gamma (t,x) \right )^{-1}_h = \widetilde D_{r}^\gamma \, (t,x)_{hr}^{-1}
\qquad \quad
\forall (t,x) \in \R \times \R^N .
\ee
\end{enumerate}

\end{lemma}

\begin{proof}
We compute directly that
\begin{align} 
 \widetilde D_r^\gamma(\tau , \zeta) \diamond_h \widetilde D_r^\gamma (t,x)  & = (r (\tau + t) , D_r^\gamma x + e^{r t A_h} D_r^\gamma \zeta ) \\
& = (r (\tau + t) , D_r^\gamma x + D_r^\gamma e^{t r S(r)^{-1} A_h S(r)} \zeta ) \\
& = \widetilde D_r^\gamma (\tau + t ,  x +  e^{t A_{hr}} \zeta ) \\
& = \widetilde D_r^\gamma \left[ (\tau , \zeta) \diamond_{hr} (t,x) \right] .
\end{align}
Then, for all $(t,x) \in \R \times \R^N$,
\begin{align}
 l_{\widetilde D_r^\gamma (\tau, \zeta)}^h \circ \widetilde D_r^\gamma (t,x) & = \widetilde D_r^\gamma (\tau, \zeta) \diamond_h \widetilde D_r^\gamma (t,x) \\
& = \widetilde D_r^\gamma \left[ (\tau , \zeta) \diamond_{hr} (t,x) \right] \\
& = \widetilde D_r^\gamma l^{hr}_{(\tau , \zeta)}  (t,x) .
\end{align}
Finally, 
\begin{align}
(0,0) &= \widetilde D_{r}^\gamma (0,0) \\
&= \widetilde D_{r}^\gamma \left [ (t,x)_{hr}^{-1} \diamond_{hr} (t,x) \right ] \\
& =    \widetilde D_{r}^\gamma (t,x)_{hr}^{-1} \diamond_{h} \widetilde D_{r}^\gamma (t,x)  .
\end{align}
Thus
\be
 \left [ \widetilde D_{r}^\gamma (t,x) \right ]^{-1}_h = \widetilde D_{r}^\gamma (t,x)_{hr}^{-1} .
\ee
\end{proof}

\subsection{Representation of Trajectories for Small $h>0$}

We note the following representation of solutions for the controlled ODE
\be
\dot \eta = A_h \eta + \beta  ,
\ee
which is inspired by the methods of \cite{Seidman, Seidman-Yong}.
It will be useful for comparing the controlled problem for small $h > 0$ to the principal part case $h=0$.

\begin{lemma} \label{lem:eta_rep}
\begin{enumerate}[(i)]
\item Let $h \in \R$ and $r \in \R \setminus \{ 0\}$. Then, for all $\tau \in [0,1]$, the flow matrix $e^{r \tau A_h}$ can be represented in the form
\be
e^{r \tau A_h} 
= S(r) \left ( e^{\tau A_0} + R_A(\tau ; hr) \right ) S(r)^{-1}
\ee
where $S$ is defined in \eqref{def:S} and the matrix $R_A(\tau \, ; h)$ is defined for all $\tau \in [0,1]$ and $h \in \R$ by
\be \label{def:RA}
R_A(\tau \, ; h) : = \sum_{i = 0}^{\kappa}  \sum_{j=0}^\kappa \sum_{m=1 \vee (j-i)}^\infty   h^{m} \frac{ \tau^{m + i -j }}{(m + i - j) !}  P_{i} A^{m + i - j} P_j .
\ee

\item Let $s, t \in \R$ with $s \neq t$, and define the interval $I = (s \wedge t, s \vee t)$.
Let $\beta \in L^1\left ( I ; \R^N \right )$. If $\eta :  I \to \R^N$ solves the ODE 
\be
\dot \eta = A_h \eta + P_0 \beta,
\ee
then
\begin{align}
\eta(t) - e^{(t-s)A_h} \eta(s) & = (t-s) S(t-s) \int_0^1 \left ( e^{\tau A_0} + R_A(\tau \, ; h (t-s))\right )P_0 \hat \beta_\tau \dd \tau \\
\end{align}
where $\hat \beta : [0,1] \to \R^N$ is defined by
\be
\hat \beta_\tau  : = \beta_{t - \tau (t-s)}
\ee

If $\beta \in L^p(I;\R^N)$ (for $p \in [1, +\infty)$), then $\hat \beta \in L^{p}((0,1); \R^N)$, with
\be
\int_I |\beta_\tau |^{p}  \dd \tau = |t-s| \int_0^1 |\hat \beta_\tau |^{p} \dd \tau .
\ee
\end{enumerate}

\end{lemma}
\begin{proof}

(i) Since $r S(r)^{-1} A_h S(r) = A_{hr}$ for all $r \neq 0$, $ S(r)  e^{r \tau A_h}  S(r)^{-1} =  e^{\tau A_{h r}} $. It therefore suffices to show that 
\be
e^{\tau A_h} 
= e^{\tau A_0} + R_A(\tau ; h) \qquad \forall h \in \R , \, \tau \in [0,1] ,
\ee
For $h = 0$ there is nothing to prove. Otherwise, since $A_h = h S(h)^{-1} A S(h)$, $e^{\tau A_h} = S(h)^{-1} e^{h \tau A} S(h) $.

Expanding the identity map into a sum of projections $I = \sum_{i=0}^\kappa P_i$, and the exponential $e^{h \tau A}$ into its power series, we find that
\be
S(h)^{-1} e^{h \tau A} S(h) = \sum_{i=0}^\kappa \sum_{j=0}^\kappa \sum_{l=0}^\infty \frac{ \tau^l}{l !}  h^{j+ l-i} P_i A^l P_j .
\ee
Recall that $\Image(A^l E_j) \subseteq \Image(A^{l+j} E_0) \subseteq V_{l+j}$, and hence $P_i A^l P_j = 0$ if $i > j + l$. Separating out the terms where $i = j+l$, we obtain
\be
S(h)^{-1} e^{h \tau A} S(h)  =   \sum_{l=0}^\kappa  \sum_{j=0}^{\kappa - l} \frac{ \tau^l}{l !} P_{j+l} A^l P_j  + \sum_{j=0}^\kappa \sum_{l=0}^\infty \sum_{i = 0}^{\kappa \wedge( j+l - 1)} \frac{ \tau^l}{l !}  h^{j+l - i} P_{i} A^l P_j .
\ee
Then, using the substitution $m = l + j - i$, we find
\be
S(h)^{-1} e^{h \tau A} S(h)  =   \sum_{l=0}^\kappa  \sum_{j=0}^{\kappa - l} \frac{ \tau^l}{l !} P_{j+l} A^l P_j  +\sum_{i = 0}^{\kappa}  \sum_{j=0}^\kappa \sum_{m=1 \vee (j-i)}^\infty   h^{m} \frac{ \tau^{m + i -j }}{(m + i - j) !}  P_{i} A^{m + i - j} P_j .
\ee

Finally, since $A_0 = \sum_{j=0}^{\kappa - 1} P_{j+1} A P_j$, we have, for all integers $l \geq \kappa$,
\be
A_0^l = \left ( \sum_{j=0}^{\kappa - 1} P_{j+1} A P_j \right)^l = \sum_{j=0}^{\kappa - l} P_{j+l} A^l P_j 
\ee
and $A_0^{\kappa + 1} =  \mb{O}_{N\times N}$.
Hence
\be
e^{\tau A_0} = \sum_{l=0}^\kappa \frac{\tau^\kappa}{l !} \sum_{j=0}^{\kappa - l} P_{j+l} A^l P_j ,
\ee
and we identify that indeed
\be
S(h)^{-1} e^{h \tau A} S(h)  = e^{\tau A_0} + R_A(\tau ; h) .
\ee

\noindent (ii) By considering the function $\tau \mapsto e^{-(\tau - s)A_h} \eta(\tau)$, which satisfies
\be
\frac{\dd}{\dd \tau} \left ( e^{-(\tau - s)A_h} \eta(\tau) \right ) = e^{-(\tau - s)A_h} P_0 \beta_\tau ,
\ee
we obtain the following expression for $\eta$:
\be
\eta(t) = e^{(t-s)A_h} \eta(s) + \int_s^t e^{(t-\tau) A_h } P_0 \beta_\tau \dd \tau .
\ee
We make the substitution $\beta_\tau  = \hat \beta_{\frac{t - \tau}{t - s}}$, and change variable $\tau ' = \frac{t - \tau}{t - s}$ to find that
\begin{align}
\eta(t) & = e^{(t-s)A_h} \eta(s) + \int_s^t e^{(t-\tau) A_h } P_0 \hat \beta_{\frac{t - \tau}{t - s}} \dd \tau \\
& = e^{(t-s)A_h} \eta(s) + (t - s) \int_0^1 e^{(t-s) \tau '  A_h } P_0 \hat \beta_{\tau ' } \dd \tau ' . \label{eq:eta_rep_1}
\end{align}
Using part (i) we have, for all $\tau \in [0,1]$,
\be
e^{(t-s) \tau A_h} = S(t-s) \left [ e^{\tau A_0} + R_A(\tau ; h (t-s) ) \right ] S(t-s)^{-1} .
\ee
We substitute this into \eqref{eq:eta_rep_1} to obtain
\begin{multline}
\eta(t) = S(t-s) \left [ e^{A_0} + R_A(1 ; h (t-s) ) \right ] S(t-s)^{-1} \eta(s) 
+ (t - s) S(t-s) \int_0^1 \left [ e^{\tau A_0} + R_A(\tau ; h (t-s) ) \right ] P_0 \hat \beta_{\tau } \dd \tau ,
\end{multline}
where we have used that $S(h)^{-1}P_0 = P_0$ for all $h \in \R  \setminus \{ 0 \}$.

\end{proof}

\subsection{Optimal Controls} \label{sec:OptControls}

In this section we look at controls that are chosen to be \emph{optimal} for a certain minimisation problem.
Let $q' = \frac{q}{q -1}$ denote the H\"older conjugate exponent of $q$.

\begin{problem} \label{prob:qopt}
Given $s \leq t$, $h \ge 0$ and $y,x \in \R^N$,
minimise
\be
\frac{1}{q'} \int_s^t |\beta_\tau|^{q '} \dd \tau \quad {\rm{subject\ to}} \quad \exists \eta: [s,t] \to \R^N : \; \;  \eta(s) = y, \; \eta(t) = x, \; \dot \eta = A_h \eta + P_0 \beta .
\ee
\end{problem}

This is a strictly convex minimisation problem with an affine constraint. The Kalman rank condition \eqref{hyp:Kalman} (imposed throughout the paper) ensures that there is at least one admissible $\beta$ satisfying the constraint, and hence there exists a unique optimal $\beta^\ast \in L^{q'}(s,t)$ for any $h \geq 0$, $s < t$ and $y,x \in \R^N$. 

Let $J_h(s,t; y,x)$ denote the value of Problem~\ref{prob:qopt}:
\begin{align} \label{def:Jr}
J_h(s,t; y,x) & := \inf \left\{\frac{1}{q'} \int_s^t |\beta_\tau|^{q '} \dd \tau \; : \; \exists \eta: [s,t] \to \R^N : \; \;  \eta(s) = y, \; \eta(t) = x, \; \dot \eta = A_h \eta + P_0 \beta \right\} \\
& = \frac{1}{q'} \int_s^t |\beta_\tau^\ast|^{q '} \dd \tau 
\end{align}

Observe from Lemma~\ref{lem:eta_rep} that
\be
 J_h(s,t; y,x) = J_h(0, t-s, 0, x - e^{(t-s) A_h} y) = : \widetilde J_h(t-s, x - e^{(t-s) A_h} y ),
\ee

The asymptotics of $\beta^\ast$ for general $h\geq 0$ and small time intervals $|t-s|$ were studied in \cite{Seidman, Seidman-Yong}.
For example, for a fixed $h \geq 0$, the following is a consequence of \cite[Theorem 1, Lemma 2.2]{Seidman-Yong}.

\begin{thm} \label{thm:Seidman-Yong}
For each $k = 0, 1, \ldots, \kappa$, there exists a norm $\| \cdot \|_k$ on the subspace $V_k$ such that the following holds. For all $\xi \in \R^N$,
\be
\lim_{t \to 0} t^{q'(k+1) - 1} \widetilde J_r(t, \xi) = \| P_k \xi \|_k^{q'}, \qquad k : = \min \{ j \in \{ 0, \ldots, \kappa \} : \xi \in V_j \}. 
\ee
\end{thm}

By scrutinising the proof of this result in \cite{Seidman-Yong}, we will be able to obtain the following estimate, in which the constants are \emph{uniform} for all sufficiently small $h$.

\begin{prop} \label{prop:tstar}
There exists $h_\ast > 0$ and a constant $C_q > 0$, depending on $A$ and $q$ only, such that for all $h \geq 0$ and $t > 0$ such that $h t \leq h_\ast$,
\be
\frac{1}{C_{q} } \, t^{- q'/q} | S(t)^{-1} \xi |^{q'} \leq \widetilde J_h(t; \xi) \leq C_q \, t^{- q'/q} | S(t)^{-1} \xi |^{q'}  \qquad \text{for all} \; \xi \in \R^N .
\ee
Consequently, for all $h \geq 0$, $s < t$ such that $h(t-s) \leq h_\ast$, and all $y,x \in \R^N$,
\be
\frac{1}{C_{q}}  (t-s)^{- q'/q} \left| S(t-s)^{-1} \left ( x - e^{(t-s) A_h} y \right ) \right|^{q'} \leq J_h(s, t; y, x) \leq C_q \, (t-s)^{- q'/q} \left| S(t-s)^{-1} \left ( x - e^{(t-s) A_h} y \right ) \right|^{q'}  .
\ee
\end{prop}

To prove Proposition~\ref{prop:tstar}, we will first reduce the problem using the transformation in Lemma~\ref{lem:eta_rep}.

\begin{lemma} \label{lem:Jtransform}
For any $\xi \in \R^N$ and $h \geq 0$, let
\be
\hat J_h(\xi) : = \inf \left \{ \frac{1}{q'} \int_0^1 |\hat \beta_\tau|^{q '} \dd \tau \; :    \; \hat\beta\in L^{q'}(0,1)\ \ {\rm{and}}\ \ \xi =    \int_0^{1} e^{\tau A_0} P_0  \hat{\beta}_{\tau} \dd \tau  +   \int_0^{1} R_A(\tau ; h ) P_0  \hat{\beta}_{\tau} \dd \tau \right \}.
\ee
We note that as a consequence of the imposed Kalman rank condition, the value of $\hat J_{h}(\xi)$ is finite, since there exists at least one continuous control $\beta : [0,1] \to \R^N$ such that the constraint is satisfied \cite[Proposition 1.1, Theorem 1.2]{Zabczyk}.

Then, for any $h \geq 0$, $t > 0$ and $\xi \in \R^N$, 
\be
\widetilde J_h(t ; \xi) = t \hat J_{h t} \left (t^{-1} S(t)^{-1} \xi \right ).
\ee

\end{lemma}
\begin{proof}
By Lemma~\ref{lem:eta_rep}, if $\dot \eta = A_h \eta + P_0 \beta$ with $ \eta(0) = 0$ and $\eta(t) = \xi$, then
\be
t^{-1} S(t)^{-1} \xi =    \int_0^{1} \left ( e^{\tau A_0} + R_A(\tau \, ; h t)  \right )P_0 \, \hat{\beta}_{\tau} \dd \tau 
\ee
where $\hat \beta_\tau  : = \beta_{t(1  - \tau)}$, and
\be
\frac{1}{q'} \int_0^t |\beta_\tau|^{q'} \dd \tau = t \; \frac{1}{q'} \int_0^1 |\hat \beta_\tau|^{q '} \dd \tau .
\ee
Since the transformation $\beta \mapsto \hat \beta$ is invertible, this completes the proof.
\end{proof}

\begin{lemma}\label{lem:norms}
There exists $C_q > 0$ (depending on $q$, $A_0$ and $P_0$ only) and $h_\ast > 0$ such that, for all $h \leq h_\ast$ and all $\xi\in\R^{N}$, $\hat J_{h}(\xi)$ is finite and we have
\be
\frac{1}{C_q} |\xi|^{q'} \leq \hat J_{h}(\xi) \leq C_q |\xi|^{q'} .
\ee
\end{lemma}

\begin{proof}

Define the operators $\mathscr{G},\mathscr{R}_h : L^{q'}(0,1) \to \RR^N$ by
\be
\mathscr{G} \zeta : = \int_0^1 e^{\tau A_0} P_0 \, \zeta(\tau) \dd \tau , 
\quad \mathscr{R}_h \zeta  : = \int_0^1 R_A (\tau ; h) \, P_0 \zeta(\tau) \dd \tau, 
\qquad \text{for all } \; \zeta \in L^{q'}(0,1).
\ee
Then the constraint on $\hat \beta$ can be written in the form
\be \label{eq:Gconstraint}
\xi = (\mathscr{G}  + \mathscr{R}_h ) \hat{\beta} .
\ee
The operator $\mathscr{G}$ can be shown to be right-invertible using \cite[Lemma 2.1]{Seidman-Yong}\footnote{In the notation of \cite{Seidman-Yong}, $\mathscr{G} = \left ( \sum_{i=0}^\kappa \frac{1}{i!} P_i \right ) \mb{M}_\kappa$.}. Thus there exists a bounded operator $\mathscr{H} : \R^N \to L^{q'}(0,1)$ such that
$\mathscr{G} \mathscr{H} \xi = \xi$ for all $\xi \in \R^N$. Hence \eqref{eq:Gconstraint} is equivalent to
\be \label{eq:Gconstraint_factored}
\xi = \mathscr{G}  ( I +\mathscr{H}  \mathscr{R}_h ) \hat{\beta}.
\ee

For all sufficiently small $h$, $ I + \mathscr{H}  \mathscr{R}_h $ defines an invertible operator $L^{q'}(0,1) \to L^{q'}(0,1)$: to see this, it suffices to show that $\| \mathscr{H}  \mathscr{R}_h \|_{\mc{B}(L^{q'}(0,1))} < 1$, or indeed, since $\mathscr{H}$ is a bounded operator $\R^N \to L^{q'}(0,1)$, that $\lim_{h \to 0} \| \mathscr{R}_h \|_{\mc{B}(L^{q'}(0,1) ; \R^N)} = 0$.
We have
\be
\| \mathscr{R}_h \|_{\mc{B}(L^{q'}(0,1) ; \R^N)} \leq \| R_A(\, \cdot \: ; h) P_0 \|_{L^{q}(0,1)} : = \left ( \int_0^1 \| R_A(\tau, h) P_0 \|_{M_{N \times N}}^q \dd \tau \right )^{1/q} .
\ee
Using the series expansion \eqref{def:RA}, we deduce that
\begin{align}
\| \mathscr{R}_h \|_{\mc{B}(L^{q'}(0,1) ; \R^N)} & \leq \sum_{l=1}^\infty  \sum_{j = 0}^\kappa \frac{h^l \| A \|^{l + j}}{(l + j)!} \left ( \int_0^1 \tau^{q(l + j)} \dd \tau \right )^{1/q} \\
& \leq \sum_{l=1}^\infty  \sum_{j = 0}^\kappa \frac{h^l \| A \|^{l + j}}{(l + j)! (1 + q(l + j) )^{1/q}} .
\end{align}
Since $(l + j)! \geq l! \, j!$, we may estimate this by
\begin{align}
\| \mathscr{R}_h \|_{\mc{B}(L^{q'}(0,1) ; \R^N)} & \leq \sum_{l=1}^\infty  \frac{ (h \| A \|)^{l}}{l !} \sum_{j = 0}^\kappa \frac{ \| A \|^{j}}{ j ! } \\
& \leq \left ( e^{h \| A \|} - 1 \right ) e^{\| A \|}.
\end{align}
Hence, defining $h_\ast$ by 
\be
h_\ast : = \frac{1}{\| A \|} \log \left ( 1 + \frac{1}{2 \| \mathscr{H} \| e^{\| A \|}} \right ),
\ee
for all $h \leq h_\ast$ we have $\| \mathscr{H}  \mathscr{R}_h \| \leq \frac{1}{2}$. Hence $ I + \mathscr{H}  \mathscr{R}_h $ is invertible with norm bounds
\be \label{est:inv_normbounds}
\| (  I + \mathscr{H}  \mathscr{R}_h )^{-1} \| \leq 2, \qquad  \| I + \mathscr{H}  \mathscr{R}_h \| \leq \frac{3}{2} .
\ee

For $h \leq h_\ast$,
we use a change of variable $\tilde \beta =  ( I + \mathscr{H} \mathscr{R}_h ) \hat \beta$ in the definition of $\hat J_{h}(\xi, h)$ to write 
\be
\hat J_{h}(\xi, h) : = \inf \left \{ \frac{1}{q'} \int_0^1 | ( I +\mathscr{H}  \mathscr{R}_h )^{-1} \tilde \beta_\tau |^{q '} \dd \tau \; {\Big \vert} \; \xi =\mathscr{G} \tilde \beta \right \} .
\ee
Note that, for all $\tilde \beta \in L^{q'}(0,1)$,
\be
\|  I +\mathscr{H}  \mathscr{R}_h  \|^{-q'} \int_0^1 |  \tilde \beta_\tau |^{q '} \dd \tau \leq \int_0^1 | ( I + \mathscr{H} \mathscr{R}_h )^{-1} \tilde \beta_\tau |^{q '} \dd \tau \leq \| ( I + \mathscr{H}  \mathscr{R}_h )^{-1} \|^{q'} \int_0^1 |  \tilde \beta_\tau |^{q '} \dd \tau ,
\ee
and hence by \eqref{est:inv_normbounds},
\be \label{est:Jtransform_equiv}
\left ( \frac{2}{3} \right )^{q'} \int_0^1 |  \tilde \beta_\tau |^{q '} \dd \tau \leq \int_0^1 | ( I + \mathscr{H}  \mathscr{R}_h )^{-1} \tilde \beta_\tau |^{q '} \dd \tau \leq 2^{q'} \int_0^1 |  \tilde \beta_\tau |^{q '} \dd \tau .
\ee

We complete the proof as in \cite{Seidman-Yong}. We first introduce a new norm on $\R^N$, induced by the operator $\mathscr{G}$. 
The quotient space $L^{q'}([0,1])/ \ker{\mathscr{G}}$ is a Banach space, for which $\mathscr{G}$ defines an isomorphism with $\R^N$.
Then the quotient norm
\be
\| \eta \|_{\mathscr{G}} : = \inf \{ \| \zeta\|_{L^{q'}} : \zeta \in L^{q'}([0,1]), \eta = \mathscr{G} \zeta \} 
\ee
defines a new norm on $\R^N$. 

Moreover, by \eqref{est:Jtransform_equiv}, for all $\tilde \beta \in L^{q'}(0,1)$ we have 
\be 
\frac{1}{C_q} \| \tilde \beta \|_{L^{q'}}^{q'} \leq \frac{1}{q'} \int_0^1 | ( I + \mathscr{H}  \mathscr{R}_h )^{-1} \tilde \beta_\tau |^{q '} \dd \tau \leq C_q \| \tilde \beta \|_{L^{q'}}^{q'} ,
\ee
and hence by taking infimum over $\tilde \beta$ such that $\xi =\mathscr{G} \tilde \beta$ we obtain
\be 
\frac{1}{C_q} \| \xi \|_{\mathscr{G}}^{q'} \leq \hat J_{h} (\xi, h) \leq C_q \| \xi \|_{\mathscr{G}}^{q'} .
\ee

Since the dimension of $\R^N$ is finite, $\| \cdot \|_{\mathscr{G}}$ is equivalent to the standard Euclidean norm. 
Hence, by modifying $C_q$, we obtain
\be 
\frac{1}{C_q} | \xi |^{q'} \leq \hat J_{h} (\xi, h) \leq C_q | \xi |^{q'} 
\ee
which completes the proof.

\end{proof}

We observe now that the proof of Proposition \ref{prop:tstar} follows immediately from Lemma \ref{lem:Jtransform} and Lemma \ref{lem:norms}.

\subsubsection{Curved Trajectories}

For equations with unbounded $f \in L^p$, our method is inspired by the approaches used in \cite{CardaliaguetGraber, Cardaliaguet-Silvestre}. In these works a `cone' of controlled trajectories is built by perturbed around optimal trajectories, created a set of non-zero measure in $\R \times \R^N$ on which $L^p$ estimates of $f$ can be used. An additional difficulty in our present setting is that, since the trajectories must follow the flow \eqref{eq:Properties_ODE}, they cannot be chosen as freely or explicitly as in these prior works. In particular, to ensure sufficient integrability of the Jacobian matrix, the trajectories need to be suitably curved near the vertex of the cone. A similar issue was considered in \cite{ADGLMR} in the setting of ultraparabolic equations. Here we build upon their construction, using perturbative arguments based on \cite{Seidman, Seidman-Yong} to extend to drift matrices $A$ that are not in principal part form (recall Definition~\ref{def:PrincipalPart}).

The following result is a consequence of those in \cite[Section 2]{ADGLMR}, transliterated into our notation.

\begin{lemma} \label{lem:ADGLMR}
Let $t > 0$ and let $(\alpha_i)_{i=0}^\kappa$ be pairwise distinct (i.e. $\alpha_i \neq \alpha_j$ for $i \neq j$).
There exist linear maps $(H_i : \R^N \to E_0)_{i=0}^\kappa$ (depending on $(\alpha_i)_{i=0}^\kappa$ and $t$) such that the following holds. 
For all $w \in \R^N$, let $\Phi( \cdot , w) : [0, t] \to \R^N$ denote the solution of the controlled ODE
\be
\partial_\tau \Phi(\tau, w) = A_0 \Phi(\tau, w) + \sum_{i=0}^\kappa \tau^{\alpha_i - 1} H_i w, \qquad \Phi(0, w) = 0, 
\ee
where we recall that $A_0$ is the principal part of $A$, defined in Definition \ref{def:PrincipalPart}.
Then 
\begin{enumerate}[(i)]
\item $\Phi(t,w) = w$ for all $w \in \R^N$.
\item For each $s \in (0, t]$, $\Phi(s, \cdot) : \R^N \to \R^N$ is an invertible linear transformation.
\item $P_0 \nabla \Phi^{-1}(s) = \sum_{i=0}^\kappa O ( s^{\alpha_i - 1})$ .
\end{enumerate}
\end{lemma}

We will extend this result to general $A$, identifying asymptotics for the associated control costs and the Jacobian of the resulting transformation.

\begin{prop} \label{prop:flows}
Let $(\alpha_i)_{i=0}^\kappa$ be pairwise distinct with $\alpha_i \in \left(\frac{1}{q}, 1\right)$ for all $i = 0, \ldots, \kappa$.
There exist positive constants $h_\ast, C>0$ such that the following holds.

Let $t > 0$ and $h \geq 0$ satisfy $h t \leq h_\ast$. Then there exists a linear map $B(t ; h) : \R^N \to E_0^{\otimes (\kappa + 1)}$ such that, letting $\beta : (0,t) \to E_0$ be defined for all $w \in \R^N$ by
\be
\beta_\tau(w) = \sum_{i=0}^\kappa \tau^{\alpha_i - 1} B_i(t ; h) w ,
\ee
and letting $\Phi(s,w)$ denote the solution of the ODE
\be
\partial_s \Phi(s,w) = A_h \Phi(s,w) + P_0 \beta_s(w), \qquad \Phi(0,w) = 0,
\ee
then $\Phi(t,w) = w$ for all $w \in \R^N$. Moreover 
\begin{enumerate}[(i)]
\item For all $s \in (0,t]$, $\Phi(s, \cdot ) : \R^N \to \R^N$ defines an invertible linear transformation.
\item \label{item:prop:flows-Lp} For all $w \in \R^N$,
\be
\int_0^t |\beta_\tau |^{q'} \dd \tau \leq C t^{- q' / q}  |S(t)^{-1} w|^{q'}
\ee
\item The Jacobian of the transformation $\Phi^{-1}(s, \cdot)$ satisfies
\be
| \det \nabla_w \Phi^{-1}(s) | \leq C 
\left ( \frac{t}{s}  \right )^{N \alpha^\ast + \sum_{j=1}^\kappa j n_j}  \qquad s \in (0, t],
\ee
where
\be \label{def:alpha_max}
\alpha^\ast : = \max_i \alpha_i ,
\ee
and the gradient satisfies
\be
|S(t)^{-1} \nabla_w \Phi^{-1}(s) S(s)| \leq C \left ( \frac{t}{s}  \right )^{\alpha^\ast} .
\ee
\end{enumerate}

\end{prop}

\begin{proof}
By Lemma~\ref{lem:eta_rep}, for every $w \in \R^N$ we are looking for $b = (b_i)_{i=0}^\kappa = \left ( b(w) \right )_{i=0}^\kappa \in E_0^{\otimes (\kappa + 1)}$ such that
\be
w = t S(t) \int_0^{1} \left ( e^{\tau A_0}  + R_A(\tau ; h t)  \right ) P_0 \hat{\beta}_{\tau} \dd \tau ,
\ee
where $\hat \beta_\tau  : = \beta_{t (1 - \tau )}$ and $\beta_\tau = \sum_{i=0}^\kappa \tau^{\alpha_i - 1} b_i $. Hence
\be
\hat \beta_\tau = \sum_{i=0}^\kappa (1 - \tau)^{\alpha_i - 1} t^{\alpha_i - 1} b_i ,
\ee
so that our constraint on $b$ may be written in the form
\be \label{eq:w_constraint}
w =  t S(t) \int_0^{1} \left ( e^{\tau A_0}  + R_A(\tau ; h t)  \right )  \sum_{i=0}^\kappa (1 - \tau)^{\alpha_i - 1} t^{\alpha_i - 1} P_0 b_i \dd \tau .
\ee

Let $\mc{D}_{\underline \alpha}(h)$ denote the family of anisotropic dilation operators on $E_0^{\otimes (\kappa + 1)}$ such that
\be
\left ( \mc{D}_{\underline \alpha}(h) b \right )_i = h^{\alpha_i} b_i \qquad \text{for } i=0, \ldots \kappa, \; \;  \text{for all } b \in E_0^{\otimes (\kappa + 1)} .
\ee
Then \eqref{eq:w_constraint} can be rewritten as
\be \label{eq:w_constraint_2}
w = S(t) \sum_{i=0}^\kappa \left ( \int_0^{1}  (1 - \tau)^{\alpha_i - 1} \left ( e^{\tau A_0}  + R_A(\tau ; h t)  \right )  \dd \tau \right ) \left ( \mc{D}_{\underline \alpha}(t) b \right )_i  
\ee

Now define operators
$\mc{G}, \mc{R}_h : (E_0)^{\otimes ( \kappa + 1)} \to \R^N$ by
\be\label{def:G_R}
\mc{G} b  := \sum_{i=0}^\kappa \int_0^1 (1-\tau)^{\alpha_i - 1} e^{\tau A_0} b_i \dd \tau , \quad \mc{R}_h b  := \sum_{i=0}^\kappa \int_0^1 (1-\tau)^{\alpha_i - 1} R_A(\tau ; h)  b_i \dd \tau \quad \text{for all } b \in (E_0)^{\otimes ( \kappa + 1)}. 
\ee
Then \eqref{eq:w_constraint_2} is the statement that
\be \label{eq:w_constraint_ops}
w = S(t) (\mc{G} + \mc{R}_{h t}) \mc{D}_{\underline \alpha}(t) b .
\ee

In \cite{ADGLMR}, it is shown that $\mc{G}$ has a right inverse $\mc{H} : \R^N \to (E_0)^{\otimes ( \kappa + 1)}$ such that $\xi = \mc{G} \mc{H} \xi$ for all $\xi \in \R^N$. We rewrite \eqref{eq:w_constraint_ops} as
\be \label{eq:w_constraint_factored}
w = S(t) \mc{G} ( I + \mc{H} \mc{R}_{h t}) \mc{D}_{\underline \alpha}(t) b .
\ee
It follows that, if the operator $( I + \mc{H} \mc{R}_{h t}) : (E_0)^{\otimes ( \kappa + 1)} \to (E_0)^{\otimes ( \kappa + 1)}$ is invertible, then
\be
b = \mc{D}_{\underline \alpha}(t)^{-1}  ( I + \mc{H} \mc{R}_{h t})^{-1} \mc{H} S(t)^{-1} w
\ee
gives a solution to the problem.

We show below in Lemma~\ref{lem:R_curved} that there exists $h_\ast$ such that, for all $h \leq h_\ast$, $I + \mc{H} \mc{R}_{h}$ is invertible with norm bounds
\be \label{est:normbounds_curved}
\| ( I + \mc{H} \mc{R}_{h} )^{-1} \| \leq 2 , \qquad \| I + \mc{H} \mc{R}_{h} \| \leq \frac{3}{2} .
\ee
Hence, for all $h t \leq h_\ast$ we may choose
\be
B(t) : = \mc{D}_{\underline \alpha}(t)^{-1}  ( I + \mc{H} \mc{R}_{h t})^{-1} \mc{H} S(t)^{-1}  .
\ee

By Lemma~\ref{lem:eta_rep}, 
\be
\int_0^t |\beta_\tau |^{q'}  \dd \tau = t \int_0^1 |\hat \beta_\tau |^{q'} \dd \tau .
\ee
Since
\begin{align}
\hat \beta_\tau &= \sum_{i=0}^\kappa (1 - \tau)^{\alpha_i - 1} t^{\alpha_i - 1} B_i(t) w \\
& = t^{ - 1} \sum_{i=0}^\kappa (1 - \tau)^{\alpha_i - 1} ( I + \mc{H} \mc{R}_{h t})^{-1} \mc{H} S(t)^{-1} w ,
\end{align}
the norm bounds \eqref{est:normbounds_curved} imply that
\be
\| \hat \beta \|_{L^{q'}} \leq t^{-1}  C |S(t)^{-1} w| \sum_{i=0}^\kappa \|  (1 - \tau)^{\alpha_i - 1} \|_{L^{q'}(0,1)} .
\ee
Since $\alpha_i > 1/q$, $(1 - \tau)^{\alpha_i - 1} \in L^{q'}(0,1)$ and we have $\| \hat \beta \|_{L^{q'}(0,1)} \leq C  t^{-1}  |S(t)^{-1} w|$. Hence
\be
\int_0^t |\beta_\tau |^{q'}  \dd \tau = t \| \hat \beta \|_{L^{q'}(0,1)}^{q'} \leq C t^{1-q'} |S(t)^{-1} w|^{q'} .
\ee
Since $q$ and $q'$ are H\"older conjugate exponents, $1 - q' = q'(\frac{1}{q'} - 1) = - \frac{q'}{q}$. We conclude that
\be
\int_0^t |\beta_\tau |^{q'}  \dd \tau  \leq C t^{- q'/ q}  |S(t)^{-1} w|^{q'} .
\ee

Then $\Phi(s,w)$ satisfies the ODE
\be
\partial_s \Phi(s,w) = A_h \Phi(s,w) + P_0 \sum_{i=0}^\kappa s^{\alpha_i - 1} B_i(t) w, \qquad \Phi(0,w) = 0 .
\ee
By Lemma~\ref{lem:eta_rep}, for any $s \in (0,t]$ we have
\begin{align}
\Phi(s,w) & = s S(s) \int_0^{1} \left ( e^{\tau A_0}  + R_A(\tau ; h s)  \right )   \sum_{i=0}^\kappa (1 - \tau)^{\alpha_i - 1} s^{\alpha_i - 1} P_0 B_i(t) w\dd \tau \\
& = S(s) \mc{G} ( I + \mc{H} \mc{R}_{h s}) \mc{D}_{\underline \alpha}(s) B(t) w \\
& = S(s) \mc{G} ( I + \mc{H} \mc{R}_{h s}) \mc{D}_{\underline \alpha} \left ( \frac{s}{t} \right) ( I + \mc{H} \mc{R}_{h t})^{-1} \mc{H} S(t)^{-1}  w .
\end{align}
Thus $\Phi(s, \cdot)$ indeed defines a linear map from $\R^N \to \R^N$ for $s \in (0,t]$. Since $s \leq t$, $h s \leq h t \leq h_\ast$, and hence $\Phi(s, \cdot)$ is invertible with inverse
\be
\Phi^{-1}(s,w) = S(t) \mc{G} ( I + \mc{H} \mc{R}_{h t}) \mc{D}_{\underline \alpha} \left ( \frac{t}{s} \right) ( I + \mc{H} \mc{R}_{h s})^{-1} \mc{H} S(s)^{-1}  w .
\ee
Therefore {the Jacobian matrix $\left ( \nabla_w \Phi^{-1} \right )_{ij} =  \nabla_{w_j} \Phi^{-1}_i \in \R^{N \times N}$} is
\be
\nabla_w \Phi^{-1}(s,w) = \left [ S(t) \mc{G} ( I + \mc{H} \mc{R}_{h t}) \mc{D}_{\underline \alpha} \left ( \frac{t}{s} \right) ( I + \mc{H} \mc{R}_{h s})^{-1} \mc{H} S(s)^{-1} \right ]^\top ;
\ee
in particular $\nabla_w \Phi^{-1}(s,w) = \nabla_w \Phi^{-1}(s)$ does not depend on $w$. 

We now define another dilation operator $\widehat{\mc{D}}_{\underline{\alpha}}$ on $E_0^{\otimes (\kappa + 1)}$ such that
\be
(\widehat{\mc{D}}_{\underline{\alpha}}(\tau) b)_i = \tau^{\alpha^\ast - \alpha_i} b_i \qquad \forall b \in E_0^{\otimes (\kappa + 1)}
\ee
(recall that $\alpha^\ast$ is defined by \eqref{def:alpha_max}).
Then $\mc{D}_{\underline{\alpha}}(\tau^{-1} ) = \tau^{- \alpha^\ast} \widehat{\mc{D}}_{\underline{\alpha}}(\tau)$ for all $\tau > 0$.
Since $\alpha^\ast \geq \alpha_i$ for all $i = 0, \ldots, \kappa$, the coefficients of $\widehat{\mc{D}}_{\underline{\alpha}}(\tau)$ are uniformly bounded for all $\tau \leq 1$.

We may therefore write 
\be
\nabla_w \Phi^{-1}(s) = \left ( \frac{t}{s}  \right )^{ \alpha^\ast} S(t) \mc{G} ( I + \mc{H} \mc{R}_{h t}) \widehat{\mc{D}}_{\underline{\alpha}} \left ( \frac{s}{t} \right) ( I + \mc{H} \mc{R}_{h s})^{-1} \mc{H} S(s)^{-1} .
\ee
Hence
\be
|\det \nabla_w \Phi^{-1}(s) | = \left ( \frac{t}{s}  \right )^{N \alpha^\ast} |\det S(t) | \left| \det S(s)^{-1}\right| \left | \det \mc{G} ( I + \mc{H} \mc{R}_{h t}) \widehat{\mc{D}}_{\underline{\alpha}} \left ( \frac{s}{t} \right) ( I + \mc{H} \mc{R}_{h s})^{-1} \mc{H} \right | .
\ee
Since $s \leq t$ and $h s \leq h t \leq h_\ast$, by \eqref{est:normbounds_curved} the matrices $\mc{G} ( I + \mc{H} \mc{R}_{h t}) \widehat{\mc{D}}_{\underline{\alpha}} \left ( \frac{s}{t} \right) ( I + \mc{H} \mc{R}_{h s})^{-1} \mc{H}$ form a bounded family. Hence there exists a constant $C>0$ such that
\be
\left | \det \mc{G} ( I + \mc{H} \mc{R}_{h t}) \widehat{\mc{D}}_{\underline{\alpha}} \left ( \frac{s}{t} \right) ( I + \mc{H} \mc{R}_{h s})^{-1} \mc{H} \right | \leq C \qquad \forall 0 < s \leq t, \; h t \leq h_\ast .
\ee
We conclude the proof by noting that 
\be
\det S(t) = t^{\sum_{j=1}^\kappa j n_j} .
\ee

\end{proof}

\begin{lemma} \label{lem:R_curved}
Let $\mc{H}$ and $\mc{R}_h$ be the operators defined in the proof of Proposition \ref{prop:flows}.
There exists $h_\ast >0$ such that, for all $h \leq h_\ast$, $I + \mc{H} \mc{R}_h$ is invertible with norm bounds
\be
\| I + \mc{H} \mc{R}_h \| \leq \frac{3}{2}, \qquad \| ( I + \mc{H} \mc{R}_h )^{-1} \| \leq 2 .
\ee

\end{lemma}
\begin{proof}
From the definition \eqref{def:G_R} of $\mc{R}_h$ and the series definition of $R_A$ \eqref{def:RA} we have, for any $b \in (E_0)^{\otimes \kappa + 1}$,
\be
\mc{R}_h b  := \sum_{i=0}^\kappa  \sum_{l = 0}^{\kappa}  \sum_{j=0}^\kappa \sum_{m=1 \vee (j-l)}^\infty   h^{m} \frac{ \tau^{m + l -j }}{(m + l - j) !} \int_0^1 (1-\tau)^{\alpha_i - 1}  P_{l} A^{m + l - j} P_j  b_i \dd \tau
\ee
Since each $b_i \in E_0$, only terms with $j=0$ are non-zero:
\be
\mc{R}_h b  :=  \sum_{m=1}^\infty   h^{m} \sum_{i=0}^\kappa  \sum_{l = 0}^{\kappa} \frac{ 1}{(m + l) !} \left ( \int_0^1 (1-\tau)^{\alpha_i - 1} \tau^{m + l } \dd \tau \right ) P_{l} A^{m + l} P_0  b_i .
\ee
Integration by parts gives
\be 
\int_0^1 (1-\tau)^{\alpha_i - 1} \tau^{m + l } \dd \tau = \alpha_i^{-1} \int_0^1 (1-\tau)^{\alpha_i} \tau^{m + l - 1} \dd \tau;
\ee
thus
\be
\mc{R}_h b =  \sum_{l=1}^\infty h^m \sum_{i=0}^\kappa \sum_{l = 0}^\kappa \frac{1}{\alpha_i (m + l)!} \left ( \int_0^1 (1-\tau)^{\alpha_i} \tau^{ m + l - 1} \dd \tau \right ) P_j A^{m + l} P_0 b_i.
\ee

The operator norm of $\mc{R}_h$ may therefore be estimated by
\be 
\| \mc{R}_h \| \leq  \sum_{m=1}^\infty h^m \sum_{i=0}^\kappa \sum_{l = 0}^\kappa \frac{1}{\alpha_i (m + l)!} \left | \int_0^1 (1-\tau)^{\alpha_i} \tau^{m + l - 1} \dd \tau \right | \| P_l A^{m + l} P_0 \| .
\ee 
Since each $\alpha_i > 0$, for all $\tau \in [0,1]$ we have $0 \leq (1-\tau)^{\alpha_i} \leq 1$. Hence
\be
 \left | \int_0^1 (1-\tau)^{\alpha_i} \tau^{m + l - 1} \dd \tau \right | \leq  \int_0^1  \tau^{m + l - 1} \dd \tau .
\ee
Then
\be
 \| \mc{R}_h \| \leq \left ( \sum_{i=0}^\kappa \frac{1}{\alpha_i} \right ) \int_0^1 \sum_{m=1}^\infty h^m  \sum_{l = 0}^\kappa \frac{1}{ (m + l)!}   \tau^{m + l - 1}  \| A  \|^{m + l} \dd \tau .
\ee
Let $j : = m-1$ to obtain 
\be
 \| \mc{R}_h \| \leq  h \| A  \| \left ( \sum_{i=0}^\kappa \frac{1}{\alpha_i} \right ) \int_0^1 \sum_{j=0}^\infty (h \tau \| A \|)^j  \sum_{l = 0}^\kappa \frac{(\tau \| A  \|)^{l}}{ (l + j + 1)!}    \dd \tau .
\ee
Then, since $(l + j + 1)! \geq j! (l + 1)!$,
\begin{align}
\| \mc{R}_h \|
& \leq h \|A \| \left ( \sum_{i=0}^\kappa \frac{1}{\alpha_i} \right ) \int_0^1 \sum_{j=0}^\infty \frac{(h \tau \| A \|)^{j}}{j!}  \sum_{l = 0}^\kappa \frac{ (\tau \|A\|^l) }{ (l + 1)!}  \dd \tau \\
& \leq h \|A \| \left ( \sum_{i=0}^\kappa \frac{1}{\alpha_i} \right ) \int_0^1 \sum_{j=0}^\infty \frac{(h \tau \| A \|)^{j}}{j!}  \sum_{l = 0}^\infty \frac{ (\tau \|A\|^l) }{ l!}  \dd \tau \\
& \leq h \|A \| \left ( \sum_{i=0}^\kappa \frac{1}{\alpha_i} \right ) \int_0^1 e^{ \tau \|A \| (1 + h) } \dd \tau \\
& \leq \frac{h}{1 + h} e^{ \|A \| (1 + h) } \left ( \sum_{i=0}^\kappa \frac{1}{\alpha_i} \right ) .
\end{align} 
Thus $\| \mc{R}_h \| \to 0$ as $h \to 0$, and so there exists $h_\ast > 0$ such that
\be
\| \mc{R}_h \| \leq \frac{1}{2 \| \mc{H} \|} \qquad \text{for all} \; h \leq h_\ast .
\ee
Then, for $h \leq h_\ast$, $\|  \mc{H} \mc{R}_h \| \leq \frac{1}{2} < 1$. The statement follows.

\end{proof}

\subsubsection{Extent of Trajectories}

The following results will allow us to estimate the maximal extent of a controlled trajectory.

\begin{lemma} \label{lem:extent}
Let $s,t \in \R$ with $s < t$.
For all $h_\ast>0$ sufficiently small,
there exists a uniform constant $C_{q}$ such that, if $h(t-s) \leq h_\ast$ and $\eta : [s,t] \to \R^N$ satisfies
\be 
\dot \eta = A_h \eta + P_0 \beta, \qquad \eta(s) = y, \; \eta(t) = x, \qquad \beta \in L^{q'}(s,t),
\ee 
then, for all $\tau \in (s,t)$,
	\begin{align}
	(t-s)^{-1/q} | S(t-s)^{-1} \left ( e^{(t-\tau) A_h} \eta(\tau) - e^{(t-s) A_h} y \right ) | & \leq C_q \| \beta \|_{L^{q'}(s,t)} \\
	(t-s)^{-1/q} | S(t-s)^{-1} \left (x - e^{(t-\tau) A_h} \eta(\tau)  \right ) | & \leq C_q  \| \beta \|_{L^{q'}(s,t)} .
	\end{align}
\end{lemma}
\begin{proof}
	
	Let $\tau \in (s,t)$. Consider $\beta_1 : = \beta \one_{(s,\tau)}$. Then $\beta_1 \in L^{q'}(s,t)$, and the path $\eta_1 : [s,t] \to \R^N$ defined as the solution of the ODE
	\be 
		\dot \eta_1 = A_h \eta_1 + P_0 \beta_1, \quad \eta_1(s) = y
	\ee 
	satisfies $\eta_1(\tau') = \eta(\tau')$ on the time interval $\tau' \in [s,\tau]$, and
	\be
	\eta_1(\tau') = e^{(\tau'-\tau) A_h} \eta_1(\tau) =  e^{(\tau'-\tau) A_h} \eta(\tau) \quad \tau ' \in [\tau, t] .
	\ee
	In particular $\eta_1(t) = e^{(t-\tau) A_h} \eta(\tau)$.
	
	It follows that
	\be
	J_r(s,t, y, e^{(t-\tau) A_h} \eta(\tau)) \leq \frac{1}{q'} \| \beta_1 \|_{L^{q'}}^{q'}
	\ee
	and hence
	\be
	\widetilde J_r (t-s, e^{(t-\tau) A_h} \eta(\tau) - e^{(t-s) A_h} y) \leq \frac{1}{q'} \| \beta \|_{L^{q'}}^{q'} .
	\ee
	
	Then, by Proposition~\ref{prop:tstar}, for all small $h_\ast > 0$ there exists $C_q > 0$ such that for all $r$ with $h (t-s) \leq h_\ast$,
	\be
	(t-s)^{-1/q} | S(t-s)^{-1} \left ( e^{(t-\tau) A_h} \eta(\tau) - e^{(t-s) A_h} y \right ) | \leq C_q \| \beta \|_{L^{q'}} 
	\ee
	
	\be
	(t-s)^{-1/q} | S(t-s)^{-1} e^{t A_h} \left ( e^{-\tau A_h} \eta(\tau) - e^{-s A_h} y \right ) | \leq C_q \| \beta \|_{L^{q'}} 
	\ee
	
	Similarly, considering the control $\beta_2 : = \beta \one_{(\tau, t)}$ and defining the path $\eta_2 : [s,t] \to \R^N$ to be the solution of the ODE
	\be
		\dot \eta_2 = A_h \eta_2 + P_0 \beta_2, \quad \eta_2(t) = x
	\ee
	produces a trajectory satisfying $\eta_2(\tau ') = \eta(\tau ')$ for all $\tau ' \in [\tau, t]$ and $\eta_2(\tau ') = e^{- (\tau - \tau ' )A_h} \eta(\tau)$ for $\tau ' \in [s, \tau]$. Thus
	\begin{align}
	(t-s)^{-1/q} | S(t-s)^{-1} \left (x - e^{(t-\tau) A_h} \eta(\tau)  \right ) | & \leq C_q J_r(s,t, e^{-(\tau - s) A_h} \eta(\tau), x)^{1/q'} \\
	& \leq C_q  \| \beta \|_{L^{q'}} .
	\end{align}
\end{proof}

\subsection{Viscosity Solutions of the Hamilton-Jacobi Equations} \label{sec:SubSuper}

In this section, we discuss comparison principles and representation formulae for viscosity solutions of Hamilton--Jacobi equations of the form
\be
\partial_t u +  \mc{H}(t, x, \nabla u) = 0
\ee
where $\mc{H}:(0,T)\times\R^N\times\R^N\to\R$ is given by
\be
 \mc{H}(t, x, \xi) : = \langle Ax, \xi \rangle + |P_0 \xi|^q + F.
\ee
Here $T>0$ is a given time horizon, the matrices $A,P_0$ are the ones governing the control system \eqref{intro:control_sys} and $ F:(0,T)\times\R^N\to\R$ is a continuous function.

A Hamiltonian $\mc{H}$ of this form satisfies the standard assumption (cf. \cite{CraIshLio})
\begin{align}
\left | \mc{H}(t, x, \frac{x-y}{\e}) - \mc{H}(s, y, \frac{x-y}{\e}) \right | \leq C_A \frac{|x-y|^2}{\e} + |F(s,y) - F(t,x)| .
\end{align}
The following local comparison principle can therefore be proved by techniques similar to \cite[Theorem II.3.7]{Bardi-CD}.

\begin{lemma} \label{lem:comparison}
Let $T>0$ be a given time horizon, let $\Omega \subset \R^N$ be a bounded domain, and let
$Q = \bigcup_{t \in [0,T]} M(t, \Omega) $, for $M : [0,T] \times \overline{\Omega}\to\R^N$ continuous such that $M(t, \cdot) : \overline{\Omega} \to \R^N$ is a bijection for all $t \in [0,T]$. 
Define the parabolic-like boundary of $Q$ by
\be 
\partial^- Q : = \{ 0 \} \times M(0, \Omega) \cup \bigcup_{t \in [0,T]} M( t, \partial \Omega ) = \partial Q \setminus \{ T \} \times M(T, \Omega) .
\ee 

Let $u_1 \in C(\overline{Q})$ and $u_2 \in C(\overline{Q})$ be, respectively, a sub- and supersolution of 
\be \label{eq:HJB_comparison}
\partial_t u +  \langle Ax, D u \rangle + \frac{\lambda^q}{q} |P_0 Du |^q +  F(t,x) = 0 \qquad (t,x) \in Q ,
\ee
where $\lambda > 0$ and $F \in C(\overline{Q})$.
Suppose that $u_1 \leq u_2$ on $\partial^- Q$. Then $u_1 \leq u_2$ on $\overline{Q}$.
\end{lemma}

\begin{lemma} \label{lem:local_value}
Let $g \in BUC(\R^N)$ and $F \in BUC([0, T] \times \R^N)$ be given bounded and uniformly continuous functions. Define the function $u$ for $(t,x) \in [0, T]\times\R^N$ by
\be
u(t,x) : = \inf_{\beta \in L^{q'}((0,t); \R^N)} \left \{ \int_{0}^t \frac{| \beta_\tau |^{q'}}{q' \lambda^{q'}}  +  F(\tau, \eta^\beta(\tau ; t,x)) \dd \tau + g (\eta^\beta(0 ; t,x)) \right \} ,
\ee
where, for each $\beta \in L^{q'}((0,t); \R^N)$, $\eta^\beta$ is the solution of the ODE
\be 
\dot \eta^\beta(s ; t,x) = A \eta^\beta(s ; t,x) + \beta_s, \; \; s \in (0,t); \qquad \eta^\beta(t) = x .
\ee
Then $u$ is a continuous viscosity solution of \eqref{eq:HJB_comparison} in $(0,T) \times \R^N$.
\end{lemma}

\section{H\"older Regularity}\label{sec:holder}

\subsection{Small Source Terms}

We begin by proving H\"older regularity in the case where the zero-order terms of the Hamilton--Jacobi equations are small, and the drift matrix is close to its principal part.

\begin{prop} \label{prop:small_Holder}

Let $p  >  N/q + 1 + \sum_{j=1}^\kappa j n_j$.
There exist $\alpha \in (0,1)$ and constants $C, \e_\ast, h_\ast > 0$ 
such that for all
$0 \leq h \leq h_\ast$ 
and $0 \leq \e  \leq \e_\ast$
the following holds. 

Let $Q^h_1$ be defined as in Definition \ref{def:Q_r}. Let $u \in C(\ov{Q}^h_1)$ be a viscosity supersolution of the Hamilton--Jacobi equation
\be
\label{eq:HJ_ImpOfOsc_sup}
\partial_t u + \left \langle A_h x ,  \nabla_x u \right \rangle + \frac{\Lambda^q}{q} | P_0 \nabla_x u |^q + \e = 0 \qquad \text{in} \; Q^h_1
\ee
and a viscosity subsolution of the Hamilton--Jacobi equation
\be
\label{eq:HJ_ImpOfOsc_sub}
\partial_t u + \left \langle A_h x , \nabla_x u \right \rangle + \frac{\lambda^q}{q} |P_0 \nabla_x u |^q -  \e f = 0 \qquad \text{in} \; Q^h_1 ,
\ee
for some non-negative continuous function $f \geq 0$ satisfying $\| f \|_{L^p} = 1$.

Suppose that $\osc_{Q_1^h} u \leq 1$. Then
\be
|u(t,x) - u(0,0)| \leq C \omega_\alpha(t,x) \qquad \forall (t,x) \in Q^h_1,
\ee
where the definition of the modulus of continuity $\omega_\alpha$ is given in Definition \ref{def:omega}.
\end{prop}

This proposition can be generalised from the cylinder $Q^h_1$ to other open sets by using the left translation operators defined in \eqref{def:LT}. We obtain the following corollary.

\begin{cor} \label{cor:HolderSmall}
Let $p  >  N/q + 1 + \sum_{j=1}^\kappa j n_j$.
There exist $\alpha \in (0,1)$ and constants $C, \e_\ast, h_\ast > 0$ 
such that for all 
$0 \leq h \leq h_\ast$ 
and $0 \leq \e  \leq \e_\ast$
the following holds. 

Let $Q^h_1$ be defined as in Definition \ref{def:Q_r}. Let $\ms{U} \subset \R \times \R^N$ be an open set and suppose that there exists a subset $\ms{V} \subset \ms{U}$ such that
\be 
\bigcup_{(t,x) \in \ms{V}} l_{(t,x)}^h \ov{Q}^h_1
\subset \ms{U} .
\ee
Let $u \in C(\ms{U})$ be a viscosity supersolution of the Hamilton--Jacobi equation
\be
\label{eq:HJ_ImpOfOsc_sup}
\partial_t u + \left \langle A_h x ,  \nabla_x u \right \rangle + \frac{\Lambda^q}{q} | P_0 \nabla_x u |^q + \e = 0 \qquad \text{in} \; \ms{U}
\ee
and a viscosity subsolution of the Hamilton--Jacobi equation
\be
\label{eq:HJ_ImpOfOsc_sub}
\partial_t u + \left \langle A_h x , \nabla_x u \right \rangle + \frac{\lambda^q}{q} |P_0 \nabla_x u |^q -  \e f = 0 \qquad \text{in} \; \ms{U} ,
\ee
for some non-negative continuous function $f \geq 0$ satisfying $\| f \|_{L^p} = 1$.

Assume that $\osc_{\ms{U}} u \leq 1$. Then, for all $(s,y), (t,x) \in \ms{V}$ with $s \leq t$,
\be
|u(t,x) - u(s,y)|  \leq C \omega_\alpha \left ( (t,x)^{-1}_h \diamond_h (s,y) \right ) .
\ee

\end{cor}

The strategy of the proof is based on the approach of \cite{Cardaliaguet-Silvestre}: we first establish an {\it improvement of oscillation} property for $u$ (Subsection~\ref{sec:ImpOsc}), by constructing suitable sub- and supersolutions for the Hamilton-Jacobi equations. Then, an iteration argument shows that $u$ is H\"older continuous (Subsection~\ref{sec:Iter}).

\subsubsection{Improvement of Oscillation} \label{sec:ImpOsc}

\begin{prop} \label{prop:ImproveOsc}
Let $p  >  N/q + 1 + \sum_{j=1}^\kappa j n_j$.
There exist constants $\e_\ast, h_\ast > 0$ and $\delta, \theta \in (0,1)$ such that for all $1/q \leq \gamma \leq 1$, $0 \leq \e \leq \e_\ast$ and $0 \leq h \leq h_\ast$ the following holds. 

Let $Q^r_1$ be defined as in Definition \ref{def:Q_r}. Let $u \in C(\ov{Q}^h_1)$ be a viscosity supersolution of the Hamilton--Jacobi equation
\be
\label{eq:HJ_ImpOfOsc_sup}
\partial_t u + \left \langle A_h x ,  \nabla_x u \right \rangle + \frac{\Lambda^q}{q} | P_0 \nabla_x u |^q + \e = 0 \qquad \text{in} \; Q^h_1
\ee
and a viscosity subsolution of the Hamilton--Jacobi equation
\be
\label{eq:HJ_ImpOfOsc_sub}
\partial_t u + \left \langle A_h x , \nabla_x u \right \rangle + \frac{\lambda^q}{q} |P_0 \nabla_x u |^q -  \e f = 0 \qquad \text{in} \; Q^h_1 ,
\ee
for some non-negative 
function $f \geq 0$ satisfying $\| f \|_{L^p} = 1$.

If
\be
0 \leq u \leq 1 \qquad \text{in} \; Q^h_1,
\ee
then
\be
\osc_{Q^{h, \gamma}_{\delta}} u \leq 1-\theta. 
\ee
\end{prop}

\begin{remark}
We give the proof here for the case where $f$ is additionally assumed to be continuous on $\ov{Q^h_1}$. The general case can be obtained by a regularisation argument: for a standard smooth compactly supported kernel $\chi_1:\R\times\R^N$ and for a family of smooth mollifiers defined as $\chi_\delta(t,x) := \delta^{-(N+1)} \chi_1(\delta^{-1}t, \delta^{-1} x)$ ($\delta > 0$), $u^\delta := \chi_\delta \ast u$ is a subsolution of 
\be
\partial_t u^\delta + \left \langle A_h x , \nabla_x u^\delta \right \rangle + \frac{\lambda^q}{q} |P_0 \nabla_x u^\delta |^q -  \e f \ast \chi_\delta -[ \div_x (Ax \chi_\delta) ] \ast u = 0
\ee
and therefore satisfies the upper bounds proved below, with the modified (continuous) source term $\e f \ast \chi_\delta + [ \div_x (Ax \chi_\delta) ] \ast u$. Observe that these bounds depend on the source only through the $L^p$ norm. Since $0 \leq u \leq 1$ by assumption, the modified source converges in $L^p$ as $\delta$ tends to zero to $\e f$, while $u^\delta$ converges pointwise to $u$ by continuity of $u$. Thus the same upper bounds for $u$ can be recovered in this case.
\end{remark}

\begin{proof}

\noindent{\bf Upper Bound.}
By the comparison principle for viscosity sub- and supersolutions (Lemma~\ref{lem:comparison}), if $u^\ast$ is a viscosity supersolution of \eqref{eq:HJ_ImpOfOsc_sub} such that $u \leq u^\ast$ on the parabolic boundary $\partial_- Q^h_{1}$ defined by
\be
\partial_- Q^h_{1} : = \{ - 1 \} \times e^{- A_h}(\Omega_{1}) \cup \bigcup_{t \in [-1, 0]} e^{t A_h} \left ( \partial \Omega_1 \right ) ,
\ee
then $u \leq u^\ast$ in $Q_{1}^h$.

By Lemma~\ref{lem:local_value}, for any $g \in BUC(\R^N)$
the following function  $u^\ast_g$, defined for $(t,x)\in [-1, 0] \times \R^N$ by
\be \label{def:u_upper}
u^\ast_g(t,x) : = \inf_{ \beta \in L^{q'}((-1,t); \R^N)
} \left \{ g(\eta^\beta (-1))+ \frac{1}{q' \lambda^{q'}} \int_{-1}^t |\beta(\tau)|^{q'} \dd \tau  + \e \int_{-1}^t f(\tau, \eta^\beta (\tau)) \dd \tau  \right \} ,
\ee

where, for any $\beta \in L^{q'}((-1,t) ; \R^N)$, $\eta^\beta$ denotes the solution of the ODE
$\dot \eta^\beta = A_h \eta^\beta + P_0 \beta$, $\eta^\beta(t) = x$,
and $f$ is extended continuously to $[-1,0] \times \R^N$ by projection onto the boundary values $f(t,x) = f\left(t, \frac{x}{|e^{-t A_h} x|}\right)$ for $|e^{-t A_h} x| > 1$,
is a viscosity supersolution of \eqref{eq:HJ_ImpOfOsc_sub}. 
We wish to choose $g$ such that $u^\ast_g \geq u$ on $\partial^-Q$; then $u \leq u^\ast_g$ will hold throughout $Q_1^h$.

\begin{figure}[h] 
\centering
\includegraphics[width=0.5\textwidth]{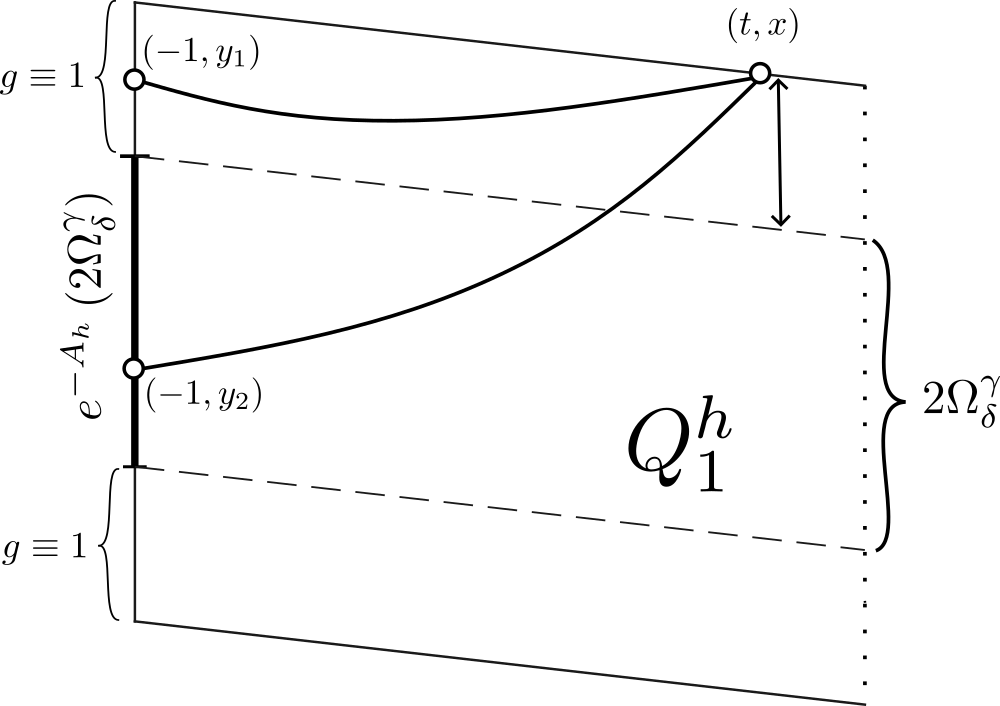}
\caption{ Estimation of the values of $u^\ast_g$ on the boundary subset $\partial^- Q^\gamma_1 \cap \{ t > -1 \}$. (i) If a controlled trajectory connects $(t,x)$ to a point $(-1, y_1)$ where $y_1 \not \in e^{-A_h} (2 \Omega_\delta)$, then the associated cost is at least $g(-1, y_1) = 1$. (ii) If a controlled trajectory connects $(t,x)$ to a point $(-1, y_2)$ where $y_2 \in e^{-A_h} (2 \Omega_\delta)$, then the associated control cost can be bounded from below, since the trajectory must cross from the spatial boundary into the inner cylinder $\bigcup_{t \in [-1, 0)} e^{t A_h} (2 \Omega_\delta^\gamma) $ (region within the dashed lines) over a time interval of length at most 1.
}
\label{fig:g_bdy}
\end{figure}

To achieve this, we begin by restricting our attention to some particular time-independent functions $g$ such that $g \equiv 1$ on $(e^{-A_h}(2 \Omega_\delta))^c$. Further properties of such $g$ functions will be given later. Then,
for any $\beta \in L^{q'}(-1, t)$ such that $\eta^\beta (-1) \not \in e^{-A_h}(2 \Omega_\delta)$ (see Figure~\ref{fig:g_bdy}),
\be \label{est:u_upper_outer}
g(\eta^\beta (-1))+ \frac{1}{q' \lambda^{q'}} \int_{-1}^t |\beta(\tau)|^{q'} \dd \tau  + \e \int_{-1}^t f(\tau, \eta^\beta (\tau)) \dd \tau \geq 1 .
\ee

On the other hand, if $\eta^\beta (-1) \in e^{-A_h}(2 \Omega_\delta)$ (Figure~\ref{fig:g_bdy}) and $(t,x) \in \partial^- Q$ with $t > -1$ (that is, $x \in e^{t A_h} \partial \Omega_1$), then by Proposition~\ref{prop:tstar}
\begin{align}
\frac{1}{q'} \int_{-1}^t |\beta(\tau)|^{q'} \dd \tau & \geq \frac{1}{C_{q}}  (1 + t)^{- q'/q} \left| S(1 + t)^{-1} \left ( x - e^{(1+t) A_h} \eta^\beta (-1) \right ) \right|^{q'} \\
& \geq \frac{1}{C_{q}}  (1 + t)^{- q'/q} \left| S(1 + t)^{-1} e^{t A_h} \left ( e^{-t A_h} x - e^{ A_h} \eta^\beta (-1) \right ) \right|^{q'} .
\end{align}
Then, since
\begin{align}
\left|  e^{-t A_h} x - e^{ A_h} \eta^\beta (-1) \right |  & \leq \| e^{-t A_h} \| \| S(1 + t)  \| \left| S(1 + t)^{-1} e^{t A_h} \left ( e^{-t A_h} x - e^{ A_h} \eta^\beta (-1) \right ) \right|
\end{align}
and, by Lemma~\ref{lem:eta_rep} (since $-1 \leq t \leq 0$),
\be
\| S(1 + t) \| = \max \{ 1, (1+t)^\kappa \}  = 1, \qquad \|  e^{-t A_h} \| = \| e^{ - t A_0} + R_A(|t|, h) \| \leq C_{h_\ast} ,
\ee
we deduce that 
\begin{align}
 \frac{1}{q'} \int_{-1}^t |\beta(\tau)|^{q'} \dd \tau &\geq C_{q, h_\ast} (1 + t)^{- q'/q} \left|  e^{-t A_h} x - e^{ A_h} \eta^\beta (-1) \right |^{q'} \\
 &\geq C_{q, h_\ast} \rm{dist}(\partial \Omega_1, 2 \Omega_\delta)^{q'}.
\end{align}
Finally, we compute $\rm{dist}(\partial \Omega_1, 2 \Omega_\delta)$: if $\tilde x \in \partial \Omega_1$ and $\tilde y \in 2 \Omega_\delta$, then 
\be
|\tilde x|^2 =\sum_{i=0}^\kappa |P_i \tilde x|^2 = 1, \qquad \sum_{i=0}^\kappa \delta^{-2 i} |P_i \tilde y|^2 \leq 4 \delta^{2\gamma} \leq 4 \delta^{\frac{2}{q}}  \Rightarrow |\tilde y|^2 =\sum_{i=0}^\kappa |P_i \tilde y|^2 \leq 4 \delta^{\frac{2}{q}}.
\ee
Hence $|\tilde x - \tilde y| \geq 1 - 2 \delta^{1/q}$. It follows that, for all $\delta \leq 4^{-q}$, $\rm{dist}(\partial \Omega_1, 2 \Omega_\delta) \geq \frac{1}{2}$. We conclude that there exists a constant $K_{q, h_\ast}$ such that, for all such $\delta$, 
\be
 \frac{1}{q'} \int_{-1}^t |\beta(\tau)|^{q'} \dd \tau \geq K_{q, h_\ast} > 0 .
\ee
Hence
\be \label{est:u_upper_inner}
g(\eta^\beta (-1))+ \frac{1}{q' \lambda^{q'}} \int_{-1}^t |\beta(\tau)|^{q'} \dd \tau  + \e \int_{-1}^t f(\tau, \eta^\beta (\tau)) \dd \tau \geq \inf_{e^{-A_h}(2 \Omega_\delta)} g + \frac{K_{q, h_\ast}}{ \lambda^{q'}}   .
\ee

Combining \eqref{est:u_upper_outer} and \eqref{est:u_upper_inner}, we deduce that
\be
u_g^\ast \geq \min \left \{ 1, \inf_{e^{-A_h}(2 \Omega_\delta)} g + \frac{K_{q, h_\ast}}{ \lambda^{q'}} \right \} \qquad \text{on }  \partial^- Q \cap \{ t > -1\} .
\ee
Since $u \leq 1$ on $\partial^- Q$ by assumption, we deduce that $u \leq u^\ast_g$ for all $g$ satisfying
\be \label{def:g}
g(x) \geq \begin{cases}
1 & x \in (e^{-A_h}(2 \Omega_\delta))^c \\
\max \left\{u(-1,x), 1 - \frac{K_{q, h_\ast}}{ \lambda^{q'}} \right\} & x \in e^{-A_h}(2 \Omega_\delta) .
\end{cases} 
\ee

Next, we will estimate  $u^\ast_g(t,x)$ from above on the smaller cylinder $Q_\delta^{h, \gamma}$ in terms of $\inf_{e^{-A_h}(2 \Omega_\delta)} g$.
To do this, we consider a particular family of controlled trajectories $\eta^\beta$, that connect a point $(- 1, y)$ with $y \in e^{-A_h} ( 2 \Omega_\delta^\gamma )$ to a point $(t,x) \in Q^{h, \gamma}_{\delta}$ (see Figure~\ref{fig:UB}).  These trajectories will be competitors in the optimisation problem defining $u^\ast_g(t,x)$ \eqref{def:u_upper}, and therefore provide an upper bound. Moreover, we construct these trajectories so as to have controls with minimal $L^{q'}$ energy. 

\begin{figure}[h]
\centering
\includegraphics[width=0.45\textwidth]{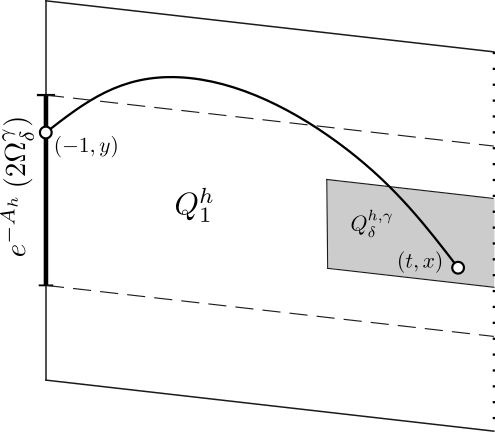}
\caption{A point $(- 1, y)$ with $y \in e^{-A_h} ( 2 \Omega_\delta^\gamma )$ can be connected to every point $(t,x) \in Q_\delta^{h,\gamma}$ (shaded region) by a controlled trajectory. If $\delta$ is sufficiently small, then trajectories with minimal control cost remain within the larger cylinder $Q^h_1$ throughout the time interval $(-1, t]$.}
\label{fig:UB}
\end{figure}

To construct the $L^{q'}$-optimal control $\beta^\ast$, recall the discussion in in Subsection~\ref{sec:OptControls}: given $y \in 2 e^{- A_h}\Omega_\delta^\gamma$ and $(t,x) \in Q^{h, \gamma}_{\delta}$, there exists a unique $\beta^\ast \in L^{q'}(- 1,   t )$ such that
\begin{align}
\frac{1}{q'} \int_{-1 }^{t} |\beta^\ast_\tau|^{q'} \dd \tau & = \min_{\beta \in L^{q'} (- 1,   t ) } \left \{ \frac{1}{q'} \int_{- 1 }^{t}  |\beta(\tau)|^{q'} \dd \tau : \eta^\beta(- 1 ) = y, \eta^\beta( t) =  x \right \} \\
& = J_{h}(- 1 ,  t ;y, e^{- t A_h} x) .
\end{align}
We will use the shorthand $\eta^\ast : = \eta^{\beta^\ast} : [-1, t] \to \R^N$ to denote the controlled trajectory satisfying
\be
\dot \eta^\ast = A_h \eta^\ast + P_0 \beta^\ast 
\quad \eta^\ast(- 1 ) = y . 
\ee

We next estimate the $L^{q'}$ norm of $\beta^\ast$ in terms of $x$ and $y$. Recall that the assumptions $(t,x) \in Q^{h, \gamma}_\delta$ and $y \in 2 e^{- A_h}\Omega_\delta^\gamma$ imply by Definitions~\ref{def:Omega} and \ref{def:Q_r} that
\be \label{eq:UB_xy_locations}
|S(\delta)^{-1} e^{- t A_h} x| < \delta^\gamma, \qquad |S(\delta)^{-1} e^{A_h} y | < 2 \delta^\gamma .
\ee
We apply Proposition~\ref{prop:tstar}: if $h (1+t) \leq h_\ast$ ( which is always true if $h \leq h_\ast$,) then 
\be
\frac{1}{q'}  \int_{-1}^{ t} |\beta^\ast_\tau|^{q'} \dd \tau \leq C_q (1 +  t)^{- q'/q} \left| S(1 +  t )^{-1}  e^{  t A_h} \left ( e^{-t A_h} x - e^{ A_h} y \right ) \right|^{q'}  .
\ee

We may write
\be
S(1 +  t )^{-1}  e^{ t A_h} \left ( e^{-t A_h} x - e^{ A_h} y \right ) = S(1 +  t )^{-1}  e^{ t A_h} S(\delta) \, S(\delta)^{-1} \left ( e^{-t A_h} x - e^{ A_h} y \right ) .
\ee
We therefore wish to estimate the operator norm of the matrix 
\be
S(1 + { t} )^{-1}  e^{  t A_h} S(\delta) = S\left( \frac{\delta}{1 +  t} \right) \,  S(\delta)^{-1} e^{ - \delta { \frac{|t|}{\delta}} A_h} S(\delta) .
\ee

Noting that $|t| \leq \delta$, we may apply Lemma~\ref{lem:eta_rep} to find that
\be \label{est:S_flow_conj}
\| S(\delta)^{-1} e^{- \delta  { \frac{|t|}{\delta}} A_h} S(\delta) \| \leq \left\| e^{\frac{|t|}{\delta}A_0} + R_{- A}\left(\frac{|t|}{\delta}, \delta h \right) \right\| \leq C_{h_\ast} .
\ee
Meanwhile, since $1- \delta \leq 1+ t \leq 1$, for all $\delta \leq \frac{1}{2}$ we have $\frac{\delta}{1 + t} \leq 1$ and hence
\be
\left \|  S \left ( \frac{\delta}{1 + t} \right )\right \| \leq 1.
\ee
We deduce that, for all $\delta \leq \frac{1}{2}$,
\be
\left \| S(1 + t )^{-1}  e^{ t A_h} S(\delta) \right \| \leq C_{h_\ast} .
\ee
It follows that
\be
 \left | S(1 +  t )^{-1}  e^{  t A_h} \left ( e^{-t A_h} x - e^{ A_h} y \right ) \right |  \leq  C_{h_\ast}\left | S(\delta)^{-1} \left ( e^{-t A_h} x - e^{ A_h} y \right ) \right | .
\ee
Hence, by \eqref{eq:UB_xy_locations},
\be \label{est:UB_x_beta_weight}
 \left | S(1 + t )^{-1}  e^{  t A_h} \left ( e^{-t A_h} x - e^{ A_h} y \right ) \right |  \leq  C_{h_\ast} \delta^\gamma .
\ee
We conclude that, for some $C_{q, h_\ast} >0$,
\be  \label{est:betaLq_opt}
\| \beta^\ast \|_{L^{q'}} \leq C_{q, h_\ast} \delta^\gamma  .
\ee

In order to obtain estimates on $u^\ast_g$ that depend on $\| f \|_{L^p(Q^h_1)}$, we need to ensure that $\eta^\ast \rvert_{(-1, t]}$ is contained in $Q^h_1$. We check that we can ensure this by taking $\delta$ sufficiently small.
We must show that, for all $\tau \in (-1, t)$, $|e^{-\tau A_h} \eta^\ast (\tau )| < 1$.
To do this, we apply Lemma~\ref{lem:extent}
to find that
\be 
 | S(1 + t)^{-1}  e^{  t A_h} \left ( e^{-t A_h} x - e^{-\tau A_h} \eta^\ast(\tau)  \right ) | \leq C_q (1 + t)^{1/q} \| \beta^\ast \|_{L^{q'}{(-1, - \delta)}} \leq  C_{q, h_\ast} \delta^\gamma ,
\ee
where in the final inequality we have used that $1 +t \leq 1$ as well as the $L^{q'}$-estimate on $\beta^\ast$ \eqref{est:betaLq_opt}.

Therefore, using \eqref{est:UB_x_beta_weight} in the case $y=0$ and the triangle inequality,
\be 
 | S(1 +  t)^{-1}  e^{ t A_h} e^{-\tau A_h} \eta^\ast(\tau)   | \leq C_q \| \beta^\ast \|_{L^{q'}} + C_{h_\ast} \delta^\gamma  \leq  C_{q, h_\ast} \delta^\gamma 
\ee
 (for a different constant $C_{q, h_\ast}$).

 It then remains only to estimate the matrix norm $\| e^{  - t A_h}  S(1 +  t)  \|$. For $\delta \leq \frac{1}{2}$ and $1-\delta \le 1+t\le 1$, we have $\frac{|t|}{1+t}\leq 1$. Then, by Lemma~\ref{lem:eta_rep},
\be
e^{  - t A_h}  S(1 + t) = S(1 + t) S(1 + t)^{-1} e^{  - t A_h}  S(1 + t) = S(1 + t) \left ( e^{\frac{|t|}{1+t} A_0} + R_A\left(\frac{|t|}{1+t}, h(1+t)\right)\right ).
\ee
Hence, since $1 + t \leq 1$, for $h \leq h_\ast$ then
\be
\| e^{  - t A_h}  S(1 +  t)  \| \leq C_{h_\ast} \| S(1 + t) \| \leq C_{h_\ast} .
\ee
Altogether, we have shown that there exists a constant $C_{q, h_\ast}'  > 0$ such that 
\be
| e^{-\tau A_h} \eta^\ast(\tau) | \leq C_{q, h_\ast}' \delta^\gamma \leq C_{q, h_\ast}' \delta^{1/q}
\ee
where we have 
used the assumptions that $\gamma \geq \frac{1}{q}$ and $\delta \leq 1$.
If $\delta < \frac{1}{(2 C_{q, h_\ast}')^q}$, then
\be \label{eq:opt_eta_contained}
| e^{-\tau A_h} \eta^\ast(\tau) | \leq \frac{1}{2}.
\ee
Hence $\eta^\ast$ is contained within $Q_1^h$.

When $f \in L^\infty$, the single path $\eta^\ast$ will suffice to bound $u^\ast_g$. However, when $f \in L^p$ for $p < + \infty$, we need an estimate for the term $\int_{-1}^t f(\tau, \eta^\beta (\tau)) \dd \tau$ that depends on the $L^p$ norm of $f$.
To obtain such a bound, we use a technique inspired by \cite{CardaliaguetGraber}: 
we construct a parametrised family of paths of the control system connecting $(-1,y)$ and $(t,x)$. By averaging over this family, we create a $(N+1)$-dimensional set over which to  integrate $f$, enabling the use of $L^p$ estimates (Figure~\ref{fig:LpCone}). In order for the argument to be valid for the widest possible range of $p$, we want to use \emph{curved} paths, to ensure that the Jacobian of the transformation is integrable near the endpoints $(-1,y)$ and $(t,x)$. In \cite{CardaliaguetGraber} the control system was $\dot \eta = \beta$, and so it was possible to write down the desired paths directly. For control systems of the form \eqref{intro:control_sys}, this is less straightforward. We will use the paths constructed in Proposition~\ref{prop:flows} (using the techniques of \cite{ADGLMR}) to define a suitable family of perturbations of $\eta^\ast$.

\begin{figure}[h]
\centering
\includegraphics[width=0.6\textwidth]{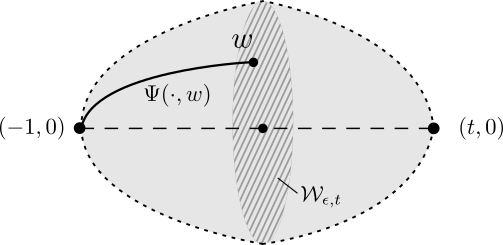}
\caption{The constant zero path from $(-1,0)$ to $(t,0)$ is perturbed to pass through $w \in \mc{W}_{\e,t}$ at the midpoint $\frac{1}{2}(t-1)$. The union of all such trajectories $\Psi(\cdot, w)$ creates a set of positive measure in $\R \times \R^N$ [region enclosed by dotted lines].}
\label{fig:LpCone}
\end{figure}

The perturbed paths are parametrised by $w \in \R^N$:  for each $w$, we add to $\eta^\ast$ a path $\Psi(\cdot, w) : [-1, t] \to \R^N$ that connects $(-1, 0)$ to the midpoint $(\frac{t-1}{2}, w)$, and then connects $(\frac{t-1}{2}, w)$ to $(t,0)$.
To construct a suitable family of paths, we apply Proposition~\ref{prop:flows}: 

First, fix $(\alpha_i)_{i=0}^\kappa \in (0, 1)$ pairwise distinct such that for all $i = 0, \ldots, \kappa$,
\be \label{def:alpha_range}
\frac{1}{q} < \alpha_i < \frac{1}{N} \left ( p - 1 - \sum_{j=0} j n_j \right ) .
\ee
This is always possible, since  $p >  N/q + 1 + \sum_{i=1}^\kappa i n_i $.
Then there exists $B_i \left ( \frac{1+t}{2} \right )$ such that the solution of
\be
\partial_\tau\Phi(\tau, w) = A_h \Phi(\tau, w) + \sum_{i=0}^\kappa (\tau + 1)^{\alpha_i - 1} B_i \left ( \frac{1 + t}{2} \right ) w, \qquad \tau \in \left[-1, \frac{t-1}{2} \right], \; \Phi(-1,w) = 0
\ee
satisfies $\Phi (\frac{t-1}{2}, w) = w$ and $\Phi(\tau,\cdot)$ defines a diffeomorphism $\R^N \to \R^N$ for $-1 < \tau \leq \frac{t-1}{2}$. Similarly there exists $\check{B}_i \left ( \frac{1+t}{2} \right )$ such that the solution of 
\be
\partial_\tau \check \Phi(\tau, w) = A_h \check \Phi(\tau, w) + \sum_{i=0}^\kappa (t - \tau)^{\alpha_i - 1} \check B_i \left ( \frac{1+t}{2} \right ) w, \qquad \tau \in \left[\frac{t-1}{2}, t \right], \; \Phi(t ,w) = 0
\ee
satisfies $\check \Phi (\frac{t-1}{2}, w) = w$ and $\check\Phi(\tau,\cdot)$ defines a diffeomorphism $\R^N \to \R^N$ for $\frac{t-1}{2} \leq \tau < t$.

We concatenate these two flows to define the controls
\be
\beta^w_\tau = 
\begin{cases}
\sum_{i=0}^\kappa (\tau + 1)^{\alpha_i - 1} B_i \left ( \frac{t + 1}{2} \right ) w  \qquad \tau \in (-1, \frac{t-1}{2} ) \\
\sum_{i=0}^\kappa (t - \tau)^{\alpha_i - 1} \check{B}_i \left ( \frac{t + 1}{2} \right ) w  \qquad \tau \in (\frac{t - 1}{2}, t ).
\end{cases}
\ee
and the flow $\Psi : (-1,t) \times \R^N \to \R^N$
\be
\Psi(\tau , w) : = \begin{cases}
\Phi(\tau, w) & \tau \in (-1, \frac{t-1}{2} ) \\
\check \Phi(\tau, w) & \tau \in (\frac{t-1}{2}, t ) .
\end{cases}
\ee
We note that, by Proposition~\ref{prop:flows},
\be\label{bound_2}
\int_{-1}^t |\beta_\tau^w |^{q'} \dd \tau = \int_{-1}^{\frac{t-1}{2}} |\beta_\tau^w |^{q'} \dd \tau +  \int_{\frac{t-1}{2}}^t |\beta_\tau^w |^{q'} \dd \tau \leq C (1+t)^{- q' / q}  |S(1+ t)^{-1} w|^{q'} .
\ee

We now define a perturbed path
\be
\eta^w (\tau) := \eta^\ast(\tau) + \Psi(\tau, w),  \qquad  \tau \in [-1, t].
\ee
Observe that $\eta^w(-1) = y$, $\eta^w(t) = x$, and $\partial_t \eta^w = A_h \eta^w + P_0 (\beta^\ast + \beta^w)$.
Hence, if $\eta^w \vert_{(-1, t])} \subset Q^r_1$, then $\eta^w$ 
is admissible as a competitor for the infimum defining $u^\ast$ (Equation \eqref{def:u_upper}).
This condition can be satisfied by taking $|w|$ sufficiently small:

Namely, consider $w \in \mc{W}_{\e, t} : = \e^{a p/N} (1 + t)^b S(1 + t) \Omega_1$ for some exponents $a \in (0,1)$ and $b > 0$ to be determined. Then, by \eqref{bound_2},
\be
\| \beta^w \|_{L^{q'}} \leq C \e^{a p/N} (1 + t)^{b - 1/q} \qquad \forall w \in\mc{W}_{\e, t} .
\ee
Hence, using Lemma~\ref{lem:extent}, since $\Psi(-1, w) = 0$ we deduce that
	\begin{align}
	(1 + t)^{-1/q} | S(1 + t)^{-1} e^{(t-\tau) A_h} \Psi (\tau, w)  | & \leq C_q  \| \beta^w \|_{L^{q'}} \\
	& \leq C_q  \e^{a p/N} (1 + t)^{b - 1/q}
	\end{align}
Then, since $0 < 1 + t < 1$,
\be
| e^{-\tau A_h} \Psi (\tau, w)  | \leq \| e^{t A_h} \| | S(1 + t)^{-1} e^{(t-\tau) A_h} \eta(\tau)  | \leq C_q  \e^{a p/N} (1 + t)^{b} \leq C_q  \e^{a p/N} .
\ee
Thus, if $\e$ is chosen small enough that $ C_q  \e^{a p/N} < \frac{1}{2}$
by \eqref{eq:opt_eta_contained} we have $\eta^w \vert_{(-1,t]} \subset Q^r_1$ for all $w \in \mc{W}_{\e,t}$.

We may therefore use the family of paths $\eta^w$ to obtain an upper bound for $u^\ast_g$:
\begin{align} \label{est:betaLq_w}
u^\ast_g & \leq \inf_{ y \in 2 e^{A_h} \Omega_\delta} \inf_{w \in \mc{W}_{\e,t}} \left \{ g(y) + \frac{\lambda^{-q'}}{q'} \int_{-1}^t |\beta^\ast_\tau + \beta^w_\tau |^{q'} \dd \tau  + \e \int_{-1}^t f(\tau, \eta^w(\tau) ) \dd \tau \right \} \\
& \leq \inf_{ y \in 2 e^{A_h} \Omega_\delta} \left \{ g(y) + \frac{\lambda^{-q'}}{q' |\mc{W}_{\e,t}|} \int_{\mc{W}_{\e,t}} \int_{-1}^t |\beta^\ast_\tau + \beta^w_\tau |^{q'} \dd \tau \dd w  + \frac{\e}{|\mc{W}_{\e,t}|} \int_{\mc{W}_{\e,t}} \int_{-1}^t f(\tau, \eta^w(\tau) ) \dd \tau \dd w \right \}.
\end{align}

By \eqref{est:betaLq_opt} and \eqref{bound_2},
\be
\| \beta^\ast + \beta^w \|_{L^{q'}} \leq C_{q, h_\ast} (\delta + (1+t)^{-1 / q}  |S(1+ t)^{-1} w| ),
\ee
and hence
\begin{multline}
u^\ast_g(t,x) \leq \inf_{ y \in 2 e^{A_h} \Omega_\delta} \left \{ g(y) + \frac{\e}{|\mc{W}_{\e,t}|} \int_{\mc{W}_{\e,t}} \int_{-1}^t f(\tau, \eta^w(\tau) ) \dd \tau \dd w \right \}\\
+ C_{q, h_\ast} \lambda^{-q'} \delta^{q'} +C_{q, h_\ast}  \frac{\lambda^{-q'}}{ |\mc{W}_{\e,t}|} \int_{\mc{W}_{\e,t}} (1+t)^{-q' / q}  |S(1+ t)^{-1} w|^{q'}  \dd w  .
\end{multline}

Since $\mc{W}_{\e, t} = \e^{a p/N} (1 + t)^b S(1 + t) \Omega_1$, using changes of variable we obtain
\begin{align}
 \frac{1}{ |\mc{W}_{\e,t}|} \int_{\mc{W}_{\e,t}} (1+t)^{-q' / q}  |S(1+ t)^{-1} w|^{q'}  \dd w & = (1+t)^{-q' / q} \frac{1}{ |\Omega_1|} \int_{\Omega_1} (\e^{a p/N} (1 + t)^b)^{q'}  |w|^{q'}  \dd w \\
 & = (1+t)^{q' ( b - 1/ q)} \e^{a p q' /N} . \label{est:curved_control_cost_scaled}
\end{align}

To estimate the term involving $f$, we use a change of variable $w' = \Psi(\tau,w)$ to write
\be
\frac{\e}{|\mc{W}_{\e,t} |} \int_{-1}^t \int_{\mc{W}_{\e,t}  }  f(\tau, \eta^{\ast}(\tau) + \Psi(\tau, w) ) \dd w \dd \tau  = \frac{\e}{|\mc{W}_{\e,t} |} \int_{-1}^t \int_{\Psi(\tau, \mc{W}_{\e,t} ) }  f(\tau, \eta^{\ast}(\tau) + w')  |\det \nabla_{w'} \Psi^{-1}(\tau, w') |  \dd w' \dd \tau .
\ee
Recalling that $\| f \|_{L^p(Q^h_1)} = 1$ by assumption,
we apply a space-time H\"older inequality to obtain
\be
\frac{\e}{|\mc{W}_{\e,t} |} \int_{-1}^t \int_{\mc{W}_{\e,t}  }  f(\tau, \eta^{\ast}(\tau) + \Psi(\tau, w) ) \dd w \dd \tau  \leq \frac{\e}{|\mc{W}_{\e,t} |}  \left ( \int_{-1}^t \int_{\Psi(\tau, \mc{W}_{\e,t})}  |\det \nabla_w \Psi^{-1}(\tau, w) |^{p'} \dd w \dd \tau \right )^{1/p'} . 
\ee
Since $\nabla_w \Psi^{-1}(\tau, w)$ is independent of $w$ ($\Psi(\tau, \cdot)$ being linear),
\begin{align}
 \frac{\e}{|\mc{W}_{\e,t} |}  &\left ( \int_{-1}^t \int_{\Psi(\tau, \mc{W}_{\e,t})}  |\det \nabla_w \Psi^{-1}(\tau) |^{p'} \dd w \dd \tau \right )^{1/p'}\\
  & \leq \frac{\e}{|\mc{W}_{\e,t} |}  \left ( \int_{-1}^t |\det \nabla_w \Psi^{-1}(\tau) |^{p'-1} \int_{\Psi(\tau, \mc{W}_{\e,t})} |\det \nabla_w \Psi^{-1}(\tau) | \dd w \dd \tau \right )^{1/p'} \\
 & \leq \frac{\e}{|\mc{W}_{\e,t} |^{1/p}}  \left ( \int_{-1}^t |\det \nabla_w \Psi^{-1}(\tau) |^{\frac{1}{p-1}} \dd \tau \right )^{1/p'}
\end{align}

By Proposition~\ref{prop:flows},
\be
 |\det \nabla_w \Psi^{-1}(\tau) | \leq \begin{cases}
 C \left ( \frac{1 + t}{1 + \tau}  \right )^{N \alpha^\ast + \sum_{j=1}^\kappa j n_j} & \tau \in \left(-1, \frac{t-1}{2} \right] \\
 C \left ( \frac{1 + t}{t - \tau}  \right )^{N \alpha^\ast + \sum_{j=1}^\kappa j n_j} & \tau \in \left[ \frac{t-1}{2} , t\right) .
 \end{cases}
\ee
Hence $|\det \nabla_w \Psi^{-1}(\tau) |^{\frac{1}{p-1}}$ is integrable when $\frac{N  \alpha^\ast + \sum_{j=0}^\kappa j n_j }{p-1} < 1$, 
or in other words when
$p > 1 + N \alpha^\ast + \sum_{j=0}^\kappa j n_j$,
which is true by \eqref{def:alpha_range}.
In this case,
\begin{align}
 \int_{-1}^t  |\det \nabla_w \Psi^{-1}(\tau) |^{\frac{1}{p-1}} \dd \tau & \leq C \left (\int_{-1}^{\frac{ t-1}{2}}  \left ( \frac{1+t}{1+\tau }  \right )^{N \alpha^\ast + \sum_{j=1}^\kappa j n_j}  \dd \tau + \int_{\frac{ t-1}{2}}^t  \left ( \frac{1+t}{t - \tau}  \right )^{N \alpha^\ast + \sum_{j=1}^\kappa j n_j}  \dd \tau \right ) \\
 & \leq C_{\alpha^\ast} (1+t) .
\end{align}
We deduce that
\be\label{est:fLp}
\frac{\e}{|\mc{W} |} \int_s^t \int_{\mc{W}  }  f(\tau, \eta^{\ast}(\tau) + \Psi(\tau, w) ) \dd w \dd \tau \leq C \e \, |\mc{W}_{\e, t}|^{-1/p} (1+t)^{1/p'} .
\ee

Next, we compute
\be
|\mc{W}_{\e,t}| = |\e^{ap/N} (1+t)^b S(1+t) \Omega_1| =  C_N \e^{a p} (1+t)^{Nb + \sum_{i=0}^\kappa i n_i}.
\ee
Hence \eqref{est:fLp} implies that
\be\label{est:fLp_scaled}
\frac{\e}{|\mc{W}_{\e,t} |} \int_{-1}^t \int_{\mc{W}  }  f(\tau, \eta^{\ast}(\tau) + \Psi(\tau, w) ) \dd w \dd \tau \leq C \e^{1 - a} \, (1+t)^{1 - (1 + Nb + \sum_{i=0}^\kappa i n_i)/p} .
\ee

Now choose the exponents $a, b$ to balance the contributions of the two terms \eqref{est:curved_control_cost_scaled} and \eqref{est:fLp_scaled}, i.e.
\be
a := \frac{N}{N + p q'} \in (0,1), \qquad b = \frac{1}{q} + \frac{p - (1 + \frac{N}{q} + \sum_{i=0}^\kappa i n_i)}{N + p q'} > \frac{1}{q}
\ee
so that 
\begin{align}
& 1 - a = \frac{a p q'}{N} = \frac{p q'}{N + p q'} = : \mu > 0, \\
& q' ( b - 1/ q) = 1 - (1 + Nb + \sum_{i=0}^\kappa i n_i)/p =  a ( p - (1 + \frac{N}{q} + \sum_{i=0}^\kappa i n_i) )  = : \nu > 0 .
\end{align}
Thus, for all sufficiently small $\delta > 0$ and $\e > 0$, we obtain the estimate
\begin{align}
u^\ast_g(t,x) & \leq \inf_{y \in 2 e^{-A_h} \Omega_\delta^\gamma }  g(y) + C_{q, h_\ast} \lambda^{-q'} \delta^{q'/q} + C \e^{\mu} (1+ t)^{ \nu } \\
& \leq \inf_{y \in 2 e^{-A_h} \Omega_\delta^\gamma } g(y) + C_{q, h_\ast} \lambda^{-q'} \delta^{q'/q} + C \e^{\mu} ,
\end{align}
where the last inequality holds since $0 < 1 + t \leq 1$ and $\nu > 0$.

\noindent{\bf Lower Bound.} 
$u$ is a viscosity supersolution of \eqref{eq:HJ_ImpOfOsc_sup}. By the comparison principle for viscosity sub- and supersolutions (Lemma~\ref{lem:comparison}), if $u_\ast$ is a viscosity subsolution of \eqref{eq:HJ_ImpOfOsc_sup}
such that $u \geq u_\ast$ on $\partial_- Q^h_{1}$, then $u \geq u_\ast$ in $Q_{1}^h$.

By Lemma~\ref{lem:local_value}, for any $\ell \in BUC(\R^N)$ the function $u_\ast^\ell$ defined by
\be
u_\ast^{\ell}(t,x) : = \inf_{\substack{\beta \in L^{q'}((-1,t); \R^N)
}} \left \{ \ell(\eta^\beta(-1 ; t,x)) + \frac{1}{q' \Lambda^{q'}} \int_{-1}^t |\beta(\tau)|^{q'} \dd \tau  - \e (1 + t) \; \right \}
\ee
is a viscosity subsolution of \eqref{eq:HJ_ImpOfOsc_sup}.

We consider $\ell \in BUC(\R^N)$ satisfying 
\be \label{def:l}
\begin{cases}
0 \leq \ell(x) \leq u(-1,x) & x \in e^{- A_h} \Omega_1 \\
\ell (x) = 0 & x \not \in e^{- A_h} \Omega_1 .
\end{cases} 
\ee
Then, by considering the zero control $\beta \equiv 0$, we obtain that for all $(t,x) \in \partial^- Q$ with $t > -1$ -- i.e. all points of the form $(t, e^{t A_h} \tilde x)$ for $\tilde x \in \partial \Omega_1$.
\be
 u_\ast^{\ell}(t,x) \leq \ell(e^{-(t + 1) A_h} x) - \e (1 + t) \leq \ell(e^{-(t + 1) A_h} x) = \ell(e^{- A_h} \tilde x) = 0  \leq u(t,x).
\ee
Hence, $u_\ast^\ell \leq u$ on $\partial^- Q$.

Thus by Lemma~\ref{lem:comparison} $u \geq u_\ast^l$ on $Q_{1}^h$, that is,  
\begin{align}
u (t,x) & \geq \inf_{\substack{\beta \in L^{q'}((-1,t); \R^N)
}} \left \{ \ell(\eta^\beta(-1 ; t,x)) + \frac{1}{q' \Lambda^{q'}} \int_{-1}^t |\beta(\tau)|^{q'} \dd \tau  - \e (1 + t) \; \right \}
\end{align}

We now wish to estimate $u_\ast^\ell$ from below on the small cylinder $Q^\gamma_\delta$. 
There are two possible cases, depending on where minimising trajectories hit the $\{t= -1\}$ boundary.

The first possibility (Figure~\ref{fig:LB_time}) is that the infimum can be approached by paths that hit $\partial_- Q^h_{1}$ at $(-1, y)$, where $e^{A_h} y \in 2 \Omega_\delta^\gamma$, i.e.
\be
u_\ast^\ell(t,x) = \inf_{y \in 2 e^{- A_h} \Omega_{\delta}^\gamma } \left \{ \ell(y) + \Lambda^{-q'} J_h(-1 ,t;y,x) - \e(1 + t) \right \} .
\ee
In this case, since $ J_{r} \geq 0$ and $1 + t \leq 1$ we estimate from below by
\be
u(t,x) \geq \inf_{y \in 2 e^{- A_h} \Omega_{\delta} } \ell(y) - \e .
\ee

\begin{figure}[h] 
\centering
\includegraphics[width=0.5\textwidth]{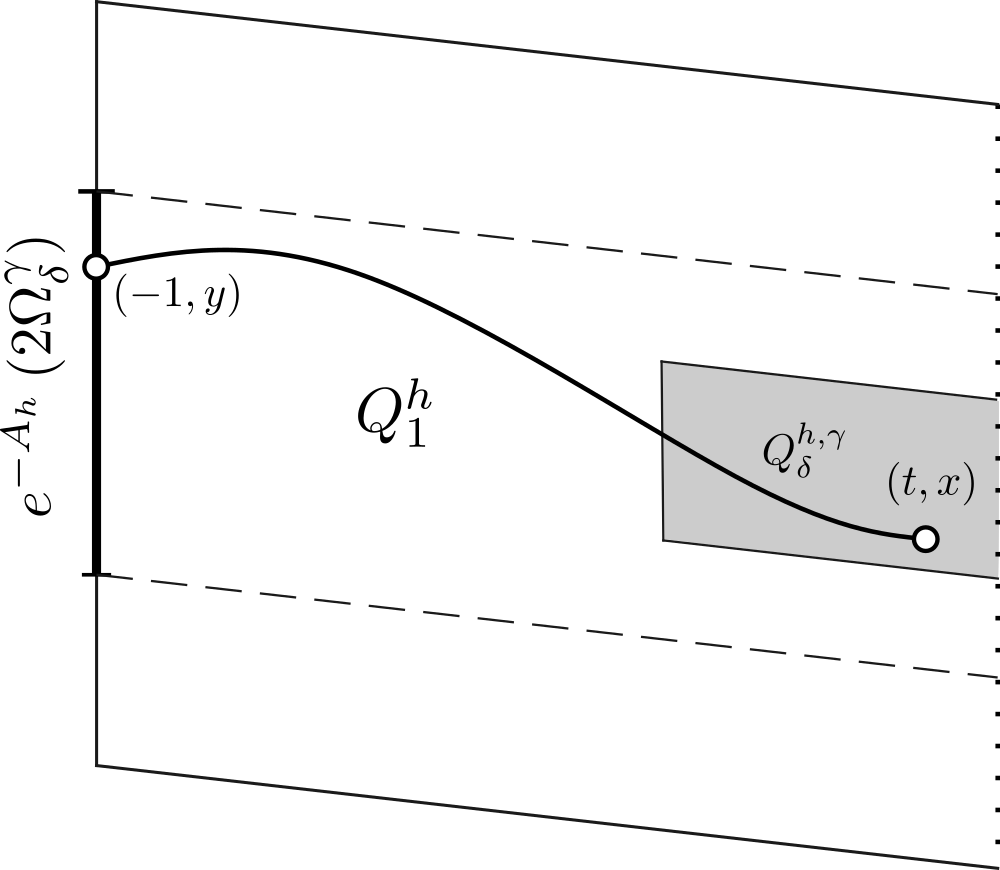}
\caption{ Optimising paths hit the boundary at points $(-1, y)$ with $y \in 2 e^{- A_h} \Omega^\gamma_\delta$ (thick line). In this case an improved lower bound is obtained if $\ell$ can be chosen to be sufficiently large everywhere on this part of the boundary while keeping $\sup_{\partial^- Q^h_1} (\ell - u) \leq 0$. This is possible provided that $\inf_{2 e^{- A_h} \Omega^\gamma_\delta} u$ is sufficiently large. If there is a point $y \in 2 e^{- A_h} \Omega^\gamma_\delta$ where this fails (i.e. $u(-1, y)$ is small), then we will instead be able to improve the upper bound (Figure~\ref{fig:UB}).}
\label{fig:LB_time}
\end{figure}

Otherwise (Figure~\ref{fig:LB_space}), we use the bound $l \geq 0$ to obtain
\be 
u_\ast^\ell(t,x) \geq \Lambda^{- q'} \inf_{y \not \in  2 e^{- A_h} \Omega_{\delta} }  J_h(-1 ,t;y,x) - \e(1 + t) .
\ee
By Proposition~\ref{prop:tstar},
\be
J_h(-1, t; y, x) \geq \frac{1}{C_{q}}  |1+t|^{- q' /q} \left | S(1+t)^{-1} \left ( x - e^{(1+t) A_h} y \right ) \right |^{q'} .
\ee

\begin{figure}[h] 
\centering
\includegraphics[width=0.5\textwidth]{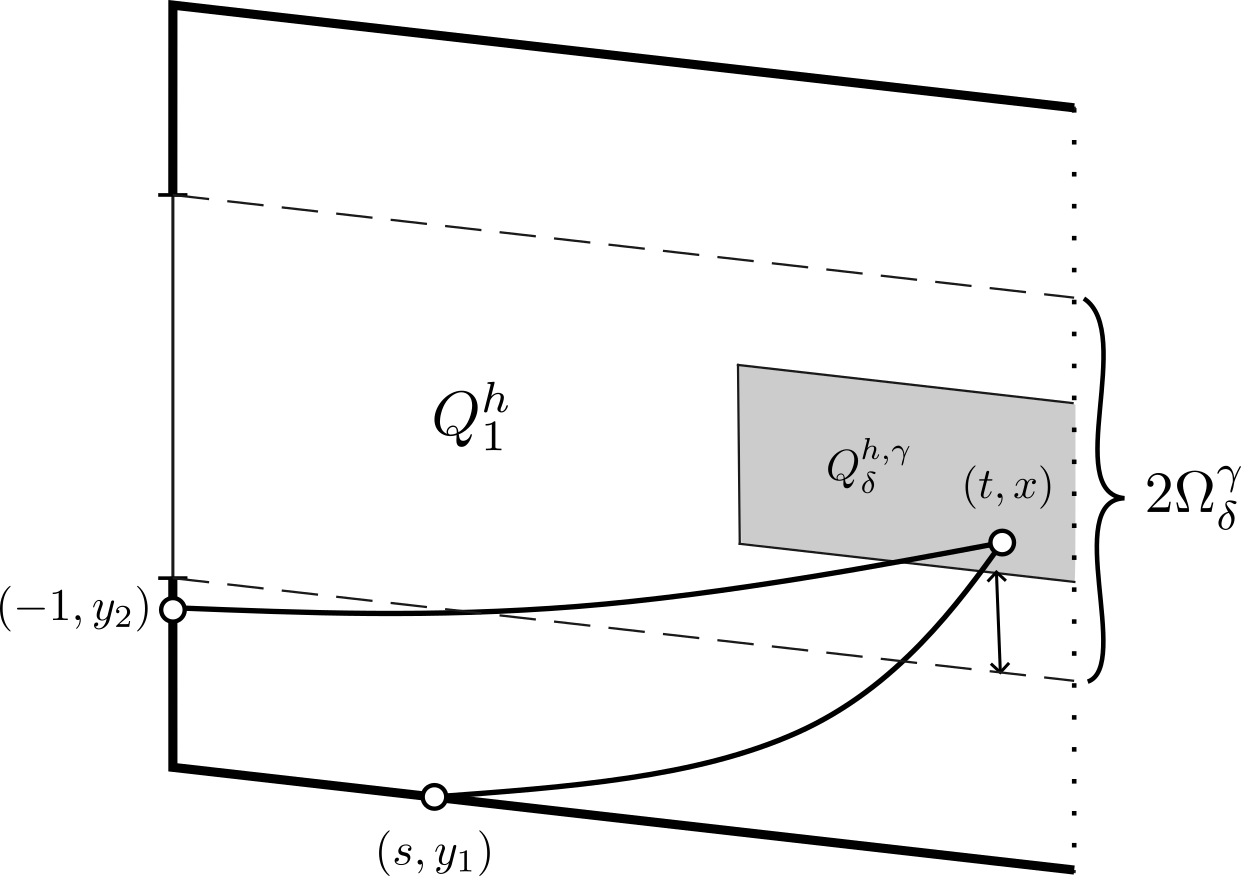}
\caption{
We consider optimising paths that either reach the spatial boundary of $Q^h_1$ at some point $y_1 \in e^{s A_h} \partial \Omega_1$, $s \in (0,t)$, or remain within $Q^h_1$ over the time interval $(-1, t)$ and hit the \emph{outer} part of the time boundary at some point $(-1,y_2)$ where $y_2 \in e^{- A_h}(\Omega_1 \setminus 2 \Omega^\gamma_\delta)$ (thick outer line). In this case the control energy is bounded below, since paths must cross from outside of $2 e^{s A_h} \Omega_\delta^\gamma$ ($s \in [-1,0]$) (dashed lines) into $Q^{h,\gamma}_\delta$ (shaded region) over a time interval of length at most 1.}
\label{fig:LB_space}
\end{figure}

There exist $\tilde y \in \Omega_1 \setminus 2 \Omega_\delta^\gamma$ and $\tilde x \in \Omega_{\delta}^\gamma$ such that $y = e^{s A_h} \tilde y$ and $x = e^{t A_h} \tilde x$.
Thus
\be \label{est:J_lower_tspt}
J_h(-1, t; y, x) \geq \frac{1}{C_{q, h_\ast}}  |1+t|^{- q' /q} \left | S(1+t)^{-1} e^{tA_h} \left ( \tilde x - \tilde y \right ) \right |^{q'} .
\ee 
Moreover, by the reverse triangle inequality
\be
| S(\delta)^{-1} (\tilde x - \tilde y) | \geq | S(\delta)^{-1}  \tilde y |  - | S(\delta)^{-1} \tilde x | \geq 2 \delta^\gamma - \delta^\gamma \geq \delta^\gamma \geq \delta,
\ee
where in the last inequality we have used $\gamma \leq 1$, $\delta \leq 1$.

It remains to estimate 
\begin{align}
\| S(\delta)^{-1} e^{- t A_h} S(1+t) \| & \leq \| S(\delta)^{-1} \| \| e^{- t A_h} \| \| S(1+t) \| \\
& \leq \delta^{- \kappa } e^{\| A_h \|} \leq C \delta^{- \kappa } ,
\end{align}
where we have used the estimate \eqref{eq:Sr_norm} for the norm of $S$, the assumption $\delta \leq 1$, and the fact that $t \in [-1, 0]$ implies that $|t|, |1+t| \leq 1$.

Hence
\be
J_h(-1, t; y, x) \geq \frac{1}{C_{q,  h_\ast}} \delta^{q'  (1 + \kappa)} .
\ee

We deduce that
\be
u(t,x) \geq \frac{1}{C_{q, h_\ast}} \Lambda^{-q'} \delta^{q'  (1 + \kappa)} - \e .
\ee

Combining the two cases, we have found that
\be
u(t,x) \geq \min \left \{\inf_{y \in 2 e^{- A_h} \Omega_{\delta}^\gamma } \ell(y) , \frac{1}{C_{q, h_\ast}} \Lambda^{-q'} \delta^{q'  (1 + \kappa )}\right \} - \e
\ee

\noindent {\bf Improvement of Oscillation.}
To summarise, so far we have obtained the following lower and upper bounds for all $(t,x) \in Q_{\delta}^{h, \gamma}$, where $\delta$ is sufficiently small: for all $g \in BUC(\R^N)$ satisfying \eqref{def:g} and $\ell\in BUC(\R^N)$ satisfying \eqref{def:l} 
\be \label{u-bounds:UL_lg}
\min \left \{\inf_{y \in 2 e^{- A_h} \Omega_{\delta}^\gamma } \ell(y) , \frac{1}{C_{q, h_\ast}} \Lambda^{-q'} \delta^{q'  (1 + \kappa )}\right \} - \e
\leq u(t,x) \leq \inf_{y \in 2 e^{- A_h} \Omega_{\delta}^\gamma }  g(y)  + C_{q, h_\ast} \lambda^{-q'} \delta^{q'/q} + C \e^\mu  .
\ee
 
We take the supremum over $\ell$ and infimum over $g$ to obtain
\begin{align} \label{u-bounds:UL}
\min \left \{\inf_{y \in 2 e^{- A_h} \Omega_{\delta}^\gamma } u(-1,y) , \frac{1}{C_{q, h_\ast}} \Lambda^{-q'} \delta^{q'  (1 + \kappa )}\right \} - \e
&\leq u(t,x)\\
& \leq \max \left\{ \inf_{y \in 2 e^{- A_h} \Omega_{\delta}^\gamma }  u(-1,y) , 1 - \frac{K_{q, h_\ast}}{\lambda^{q'}}  \right\} + C_{q, h_\ast} \lambda^{-q'} \delta^{q'/q} + C \e^\mu\nonumber  ,
\end{align}
where $K_{q, h_\ast} > 0$ is a constant.

We now wish to choose the parameters $\delta$, $\theta$ and $\e_\ast$ to satisfy the following:
\begin{align} \label{est:lambda}
C_{q, h_\ast} \lambda^{-q'} \delta^{q'/q} & \leq  \min \left\{ 1 - 2 \theta,  \frac{K_{q, h_\ast}}{\lambda^{q'}} \right\}  - 2 \theta \\ \label{est:Lambda}
 \frac{1}{C_{q, h_\ast}} \Lambda^{-q'} \delta^{q'  (1 + \kappa )} &\geq 2 \theta \\
\e_\ast  \leq \theta , \; C \e^\mu_\ast & \leq \theta . \label{est:eps}
\end{align}
This is possible if we first choose $\delta > 0$ small enough that the estimates \eqref{u-bounds:UL} hold and 
\be 
C_{q, h_\ast} \delta^{q'/q} < K_{q, h_\ast} .
\ee
Then fix $\theta > 0$ small enough that, for the given value of $\delta$,
\be
0 < \theta <  \frac{1}{2} \min \left \{  \frac{K_{q, h_\ast}}{\lambda^{q'}}  - C_{q, h_\ast} \lambda^{-q'} \delta^{q'/q}, \frac{1}{2}(1 - C_{q, h_\ast} \lambda^{-q'} \delta^{q'/q}), \frac{1}{ C_{q, h_\ast}} \Lambda^{-q'} \delta^{q'  (1 + \kappa)}  \right \} ;
\ee
this is possible since the right hand side is strictly positive. This ensures that \eqref{est:lambda}-\eqref{est:Lambda} are satisfied. Finally, take $\e_\ast > 0$ small enough that the estimates \eqref{u-bounds:UL} hold and \eqref{est:eps} is satisfied for the given value of $\theta$.

Then, for the chosen $\theta, \delta$ and any $\e \leq \e_\ast$, from \eqref{u-bounds:UL} and \eqref{est:lambda}-\eqref{est:Lambda}-\eqref{est:eps} we obtain 
\be \label{u-bounds:UL_params}
\min \left \{\inf_{y \in 2 e^{- A_h} \Omega_{\delta} } u(-1 ,y) , 2 \theta \right \} - \theta
\leq u(t,x) \leq \max \left\{ \inf_{y \in 2 e^{- A_h} \Omega_{\delta}^\gamma }  u(-1,y) ,  1- \frac{K_{q, h_\ast}}{\lambda^{q'}}   \right\}   -  \max \left\{ 2 \theta, 1- \frac{K_{q, h_\ast}}{\lambda^{q'}} \right\}  + 1-  \theta.
\ee

Then either:
\begin{enumerate}[(i)]
\item $\ds \inf_{y \in 2 e^{- A_h} \Omega_{\delta}^\gamma } u(-1 ,y) \leq 2 \theta$, in which case the upper bound of \eqref{u-bounds:UL_params} gives $\sup_{Q^{h, \gamma}_\delta} u \leq 1-\theta$, or
\item $\ds \inf_{y \in 2 e^{- A_h} \Omega_{\delta}^\gamma } u(-1 ,y) > 2 \theta$, in which case the lower bound of \eqref{u-bounds:UL_params} gives $\inf_{Q^{h, \gamma}_\delta} u \geq \theta$.
\end{enumerate}  
Since by assumption $0 \leq u \leq 1$ on $Q^h_1 \supset Q^{h, \gamma}_\delta$, in either case we deduce that
\be
\osc_{Q^{h, \gamma}_\delta} u \leq 1- \theta.
\ee

\end{proof}

\subsubsection{H\"{o}lder Regularity} \label{sec:Iter}

In this section, we complete the proofs of Proposition~\ref{prop:small_Holder} and Corollary~\ref{cor:HolderSmall}. The key step is to use the improvement of oscillation to prove H\"older regularity. The proof is based on a standard iteration argument (see e.g. \cite{Cardaliaguet-Silvestre}); however, in our present setting, we must take care to pay attention to the effect of the rescaling $\widetilde D_{\rho}$ on the drift matrix $A_h$.

\begin{lemma} \label{lem:OscIter}
Let $p, \theta, \delta, \e_\ast$ and $h_\ast$ be as in Proposition~\ref{prop:ImproveOsc}. Define 
\be \label{def:alpha_Holder}
\alpha : = \min \left \{ \frac{\log(1-\theta)}{\log \delta}, \frac{1 -  \frac{1}{p} (\frac{N}{ q} + 1 + \sum_{j=1}^\kappa j n_j )}{1 + \frac{N}{ p q '} } \right \}, \quad \gamma : = \frac{1}{q} + \frac{\alpha}{q'} .
\ee

Let $Q^{h, \gamma}_\rho$ be defined as in Definition \ref{def:Q_r} for this value of $\gamma$ and any $0\leq h \leq h_\ast$. 

Let $u \in C(\ov{Q}^h_1)$ be a viscosity supersolution of the Hamilton--Jacobi equation
\be
\label{eq:HJ_HR_sup}
\partial_t u + \left \langle A_h x ,  \nabla_x u \right \rangle + \frac{\Lambda^q}{q} | P_0 \nabla_x u |^q + \e = 0 \qquad \text{in} \; Q^h_1
\ee
and a viscosity subsolution of the Hamilton--Jacobi equation
\be
\label{eq:HJ_HR_sub}
\partial_t u + \left \langle A_h x , \nabla_x u \right \rangle + \frac{\lambda^q}{q} |P_0 \nabla_x u |^q -  \e f = 0 \qquad \text{in} \; Q^h_1 ,
\ee
for $0 \leq h \leq h_\ast$, $\e \leq \e_\ast$ and some non-negative continuous function $f \geq 0$ satisfying $\| f \|_{L^p(Q^h_1)} = 1$.

Suppose that $\osc_{Q^h_1} u \leq 1$. Then, for all $\rho \in (0,1)$,
\be
\osc_{Q^{h, \gamma}_{\rho}} u \leq \delta^{-\alpha} \rho^\alpha .
\ee
\end{lemma}

\begin{proof}
The proof is by induction. Let 
\be u_0 : = u , \qquad u_{n+1} : = \delta^{-\alpha} u_n \circ \widetilde D_\delta^\gamma = \delta^{- \alpha n} u \circ \widetilde D_{\delta^n}^\gamma,
\ee
where we have used the fact that $( \widetilde D_\delta^\gamma )^n = \widetilde D_{\delta^n}^\gamma$.
Recalling Remark~\ref{rmk:CylinderTransformation} we note further that $( \widetilde D_{\delta^n}^\gamma)^{-1} Q^h_1  = Q^{\delta^n h, \gamma}_{ \delta^{-n}} \supset Q^{\delta^n h}_{ 1}$, and so $u_n$ is well-defined at least on the set $Q^{\delta^n h}_{ 1}$.

From the rescaling \eqref{eq:f_PRescalingCost} we have moreover that
$u_n$ is a viscosity supersolution of 
\be
\label{eq:HJ_ImpOfOsc_sup_n}
\partial_t u_n + \left \langle A_{ \delta^n h} x ,  \nabla_x u_n \right \rangle + \frac{\Lambda^q}{q} | P_0 \nabla_x u_n |^q + \delta^{n (1-\alpha)} \e = 0 \qquad \text{in} \; \; Q^{\delta^n h}_1
\ee
and a viscosity subsolution of 
\be
\label{eq:HJ_ImpOfOsc_sub_n}
\partial_t u_n + \left \langle A_{\delta^n h} x , \nabla_x u_n \right \rangle + \frac{\lambda^q}{q} |P_0 \nabla_x u_n |^q -  \e  \delta^{  n \left ( 1 -  \frac{1}{p} (\frac{N}{ q} + 1 + \sum_{i=1}^\kappa i n_i )  - \alpha \left (1 + \frac{N}{ p q '} \right ) \right ) }  f_{\delta^n} = 0 \qquad \text{in} \; Q^{\delta^n h}_1 ,
\ee
where $\| f_{\delta^n} \|_{L^p } \leq 1$
Due to the choice of $\alpha$ \eqref{def:alpha_Holder}, the exponents of $\delta$ in each of the source terms are non-negative. Since $\delta < 1$, $u_n$ therefore satisfies the assumptions of Proposition~\ref{prop:ImproveOsc}.

As the inductive step, assume that
\be
\osc_{Q^{\delta^n h}_1} u_n \leq 1.
\ee
Then, by Proposition~\ref{prop:ImproveOsc},
\be
\osc_{Q^{\delta^n h, \gamma}_\delta} u_n \leq 1 - \theta.
\ee
Meanwhile, by definition of $u_{n+1}$ we have
\be 
\osc_{ Q^{\delta^{n+1} h}_1} u_{n+1} = \delta^{-\alpha} \osc_{ Q^{\delta^{n+1} h}_1} u_n \circ \widetilde D_\delta^\gamma
\ee
By Remark~\ref{rmk:CylinderTransformation}, $\widetilde D_\delta^\gamma Q^{\delta^{n+1} h}_1 = Q^{\delta^{n} h, \gamma}_\delta$.
Hence
\be
\osc_{ Q^{\delta^{n+1} h}_1} u_{n+1} = \delta^{-\alpha} \osc_{Q^{\delta^{n} h, \gamma}_\delta} u_n \leq \delta^{-\alpha} (1 - \theta).
\ee
By the choice of $\alpha$ \eqref{def:alpha_Holder}, $\delta^{-\alpha} (1 - \theta) \leq 1$. Thus we have obtained 
\be
\osc_{ Q^{\delta^{n+1} h}_1} u_{n+1}  \leq 1 
\ee
as required.

Therefore, by induction, for all integer $n \geq 0$ we have 
\be
\osc_{Q^{\delta^n h}_1} u_n \leq 1.
\ee
That is, 
\be
\osc_{Q^{\delta^n h}_1} \delta^{-\alpha n} u \circ ( \widetilde D_\delta^\gamma )^n \leq 1.
\ee
Hence, using Remark~\ref{rmk:CylinderTransformation} once more, 
\be
\osc_{Q^{h, \gamma}_{\delta^n}} u = \osc_{\widetilde D_{\delta^n} Q^{\delta^n h}_1}  u  \leq \delta^{\alpha n}.
\ee

For general $\rho > 0$, find $n$ such that $\delta^{n+1} \leq \rho \leq \delta^n$. Then 
\be
\osc_{Q^{h, \gamma}_\rho} u \leq \osc_{Q^{h, \gamma}_{\delta^n}} u \leq \delta^{ \alpha n} \leq \delta^{-\alpha} \delta^{ \alpha (n+1)} \leq \delta^{-\alpha} \rho^\alpha .
\ee
The proof is complete.

\end{proof}

\begin{lemma} \label{lem:inf_rho}
Let $\gamma \in [1/q, 1]$ and $0 \leq h < h_\ast$ be given. For all $t \leq 0$ and $x \in \R^{N}$ define
\be \label{def:inf_rho}
\rho(t,x) : = \inf \{ \rho' > 0 : (t,x) \in \ov{Q}^{h, \gamma}_{\rho'} \} .
\ee
Then there exists a constant $C>0$ depending on $h_\ast$ and $q$ only such that
\be
\rho(t,x) \wedge 1 \leq C \left ( |t| + \sum_{j=0}^\kappa |P_j x|^{\frac{1}{\gamma + j}} \right ) .
\ee
\end{lemma}
\begin{proof}
Recall that by definition, $(t,x) \in \ov{Q}^{h, \gamma}_{\rho'}$ if and only if $- \rho' \leq t \leq 0$ and $|D_{1/\rho'}^{(\gamma)} e^{-t A_h} x | \leq 1$.
Consider, for some $R>1$ to be determined, $\rho_R$ defined by
\be
\rho_R : = R  \left ( |t| + \sum_{j=0}^\kappa |P_j x|^{\frac{1}{\gamma + j }} \right ) .
\ee
We claim that, 
if $R$ is large enough (independently of $(t,x)$) and $\rho_R < 1$, 
then $(t,x) \in \ov{Q}^r_{\rho_R}$; this implies that $\rho \leq \rho_R$ for $\rho_R < 1$.
If $\rho_R \geq 1$, then there is nothing to prove.

Since $R>1$ by assumption,
$\rho_R > |t|$. Hence indeed $- \rho_R \leq t \leq 0$, since $t$ was assumed non-positive.
It remains to estimate $|D_{1/\rho_R}^{(\gamma)} e^{-t A_h} x |$.

First observe that
\be
D_{1/\rho_R}^{(\gamma)} e^{- t A_h} = e^{- t \rho_R^{-1} A_{h \rho_R}} D_{1/\rho_R}^{(\gamma)} .
\ee
By Lemma~\ref{lem:eta_rep},
\be
\| e^{- t \rho_R^{-1} A_{h \rho_R}} \| \leq \| e^{A_0} + R_{A}(|t| \rho_R^{-1} ; h \rho_R) \| .
\ee
Since $|t| \rho_R^{-1} \leq 1$ and $h \rho_R \leq h_\ast$, there exists a constant $C>0$ depending on $h_\ast$ only such that
\be
\| e^{-t \rho_R^{-1} A_{h \rho_R}} \| \leq C .
\ee
It follows that
\be
|D_{1/\rho_R}^{(\gamma)} e^{-t A_h} x | \leq C | D_{1/\rho_R}^{(\gamma)} x| .
\ee

To estimate 
\be
| D_{1/\rho_R}^{(\gamma)} x | = \left ( \sum_{j=0}^\kappa \rho_R^{- 2 (\gamma + j )} |P_j x|^2 \right )^{1/2},
\ee
observe that, for each $j=0, \ldots, \kappa$,
\be
\rho_R \geq R |P_j x|^{\frac{1}{\gamma + j }} .
\ee
Hence
\be
\rho_R^{- 2 (\gamma  + j )} |P_j x|^2 \leq R^{-2(\gamma + j)} ,
\ee
whence it follows that
\be
| D_{1/\rho_R}^{(\gamma)} x | = R^{-\gamma} \left ( \sum_{j=0}^\kappa R^{- 2  j }  \right )^{1/2}.
\ee
Since $R>1$, we can obtain the following estimate which is independent of $\gamma \in [\frac{1}{q}, 1]$:
\be
| D_{1/\rho_R}^{(\gamma)} x | \leq R^{-1/q} \left ( \sum_{j=0}^\kappa R^{- 2  j }  \right )^{1/2}.
\ee

We have therefore shown that
\be
|D_{1/\rho_R}^{(\gamma)} e^{t A_h} x | \leq C_{h_\ast} R^{-1/q} \left ( \sum_{j=0}^\kappa R^{- 2  j }  \right )^{1/2} .
\ee
Thus, for $R$ large enough that 
\be
C_{h_\ast} R^{-1/q} \left ( \sum_{j=0}^\kappa R^{- 2  j }  \right )^{1/2} < 1,
\ee
we have 
\be
|D_{1/\rho_R}^{(\gamma)} e^{- t A_h} x | < 1 ;
\ee
in other words, $(t,x) \in Q^{h, \gamma}_{\rho_R}$. Notice that the choice of $R$ depends only on $C_{h_\ast}$ and $q$, and hence only on $h_\ast$ and $q$. It follows that $\rho < \rho_R$, which completes the proof.

\end{proof}

\begin{proof}[Proof of Proposition~\ref{prop:small_Holder}]
Let $\alpha, \gamma$ be as in \eqref{def:alpha_Holder}, and let $\rho = \rho(t,x)$ be defined by \eqref{def:inf_rho} for this value of $\gamma$. Since $(t,x) \in \ov{Q}^h_1$ by assumption, $\rho \leq 1$. Hence, by Lemma~\ref{lem:inf_rho} there exists $C>0$ such that
\be \label{eq:rho_bound_smallHolder}
\rho \leq C \left ( |t| + \sum_{j=0}^\kappa |P_j x|^{\frac{1}{\gamma + j}} \right ) .
\ee

Meanwhile, by Lemma~\ref{lem:OscIter}, for some constant $C>0$,
\be
\osc_{\ov{Q}^{h, \gamma}_{\rho}} u \leq C \rho^\alpha .
\ee
Since $(t,x), (0,0) \in \ov{Q}^{h, \gamma}_{\rho}$, we have in particular that
\be
|u(t,x) - u(0,0)| \leq \osc_{\ov{Q}^{h, \gamma}_{\rho}} u \leq C \rho^\alpha .
\ee
Hence, by \eqref{eq:rho_bound_smallHolder},
\begin{align}
|u(t,x) - u(0,0)| & \leq C \left ( |t| + \sum_{j=0}^\kappa |P_j x|^{\frac{1}{\gamma + j}} \right )^\alpha \\
& \leq C \left ( |t|^\alpha + \sum_{j=0}^\kappa |P_j x|^{\frac{\alpha}{\gamma + j}} \right ) .
\end{align}
Finally, substituting $\gamma = \frac{1}{q} + \frac{\alpha}{q'}$ completes the proof.

\end{proof}

\begin{proof}[Proof of Corollary~\ref{cor:HolderSmall}]
	Given $(s,y)$ and $(t,x)$ with $s \leq t$, 
Then $u \circ l_{(t,x)}^h$ satisfies the assumptions of Proposition~\ref{prop:small_Holder}, while
\begin{align}
|u(t,x) - u(s,y)| &= |u \circ l_{(t,x)}^h(0,0) - u \circ l_{(t,x)}^h ((t,x)^{-1}_h \diamond_h (s,y)) | 
\end{align}
If $(t,x)^{-1}_h \diamond_h (s,y) \in \ov{Q}^h_1$ then by Proposition~\ref{prop:small_Holder}, 
\be
 |u \circ l_{(t,x)}^h(0,0) - u \circ l_{(t,x)}^h ((t,x)^{-1}_h \diamond_h (s,y)) |\leq C \omega_\alpha((t,x)^{-1}_h \diamond_h (s,y)) .
\ee
Otherwise $(t,x)^{-1}_h \diamond_h (s,y) \not \in \ov{Q}^h_1$, in which case 
\be
\rho((t,x)^{-1}_h \diamond_h (s,y)) \geq 1,
\ee
where $\rho$ is the function defined in \eqref{def:inf_rho}.
Then, by Lemma~\ref{lem:inf_rho}, there exists a constant $C > 0$ such that
\be
1 \leq  C \omega_\alpha((t,x)^{-1}_h \diamond_h (s,y)) .
\ee
Then, since $\osc_{\ms{U}} u \leq 1$ by assumption,
\begin{align}
 |u \circ l_{(t,x)}^h(0,0) - u \circ l_{(t,x)}^h ((t,x)^{-1}_h \diamond_h (s,y)) | & \leq 1 \\
 & \leq C \omega_\alpha((t,x)^{-1}_h \diamond_h (s,y)) .
\end{align}
This completes the proof.
\end{proof}

\subsection{General Case}\label{subsec:general}

To complete the proof of Theorem~\ref{thm:main} in the general case, we perform a final rescaling that reduces the problem to the setting of Corollary~\ref{cor:HolderSmall}.

\begin{proof}[Proof of Theorem~\ref{thm:main}]
Fix a compact set $K \subset \ms{U}$.
We perform a translation and rescaling to obtain a function that satisfies the assumptions of Corollary~\ref{cor:HolderSmall}. First, let $\hat u = u/(2 \| u \|_{L^\infty})$; this has $\osc_{\ms{U}} \hat u \leq 1$ and is a viscosity supersolution of
\be \label{eq:HJ_geq_OscOne}
\partial_t \hat u + \left \langle A x , \nabla_x \hat u \right \rangle + \frac{\Lambda^q (2 \| u \|_{L^\infty})^{q-1}}{q} |P_0 \nabla_x \hat u |^q + \frac{c_0}{2 \| u \|_{L^\infty}} = 0 
\ee
and a viscosity subsolution of
\begin{align}\label{eq:HJ_leq_OscOne}
\partial_t \hat u + \left \langle A x , \nabla_x \hat u \right \rangle + \frac{\lambda^q (2 \| u \|_{L^\infty})^{q-1}}{q} |P_0 \nabla_x \hat u |^q - \frac{f}{2 \| u \|_{L^\infty}} = 0 
\end{align}
on $\ms{U}$.

Then let $\e_\ast, h_\ast, \theta, \delta$ be the parameters from Proposition~\ref{prop:ImproveOsc} corresponding to inequalities \eqref{eq:HJ_geq_OscOne}-\eqref{eq:HJ_leq_OscOne}; these depend on $A$, $q, \Lambda, \lambda$ and $\| u \|_{L^\infty}$.
We then use the dilations $\widetilde D_\rho^{(\frac{1}{q})}$ to transform \eqref{eq:HJ_geq_OscOne}-\eqref{eq:HJ_leq_OscOne} into inequalities with sufficiently small source terms and $h \leq h_\ast$, such that Corollary~\ref{cor:HolderSmall} is applicable.

First fix, an open set $\ms{V}$ such that $\overline{\ms{V}}$ is compact and $\overline{\ms{V}} \subset \ms{U}$. Then $f_+ \in L^p(\ms{V})$ and $\sup_{\ms{V}} c_0$ is finite. On $\ms{V}$, $u$ is a viscosity supersolution of
\be \label{eq:HJ_geq_OscOne_loc}
\partial_t \hat u + \left \langle A x , \nabla_x \hat u \right \rangle + \frac{\Lambda^q (2 \| u \|_{L^\infty})^{q-1}}{q} |P_0 \nabla_x \hat u |^q + \frac{\sup_{\ms{V}} c_0}{2 \| u \|_{L^\infty}} = 0 
\ee
and a viscosity subsolution of
\begin{align}\label{eq:HJ_leq_OscOne_loc}
\partial_t \hat u + \left \langle A x , \nabla_x \hat u \right \rangle + \frac{\lambda^q (2 \| u \|_{L^\infty})^{q-1}}{q} |P_0 \nabla_x \hat u |^q - \frac{f_+}{2 \| u \|_{L^\infty}} = 0 .
\end{align}
Let
\be
\hat u_r = \hat u \circ \widetilde D_r^{(\frac{1}{q})} ,
\ee
for some $r \in (0,1)$ to be determined. By Corollary~\ref{rmk:HJ_rescaled_Ar} (using the parameters $\gamma = \frac{1}{q}$, $\alpha = 0$), $\hat u_r$ is a viscosity supersolution of
\be
\partial_t \hat u_r + \left \langle A_r x ,  \nabla_x \hat u_r \right \rangle + \frac{\Lambda^q (2 \| u \|_{L^\infty})^{q-1}}{q} | P_0 \nabla_x \hat u_r |^q + r \frac{ \sup_{\ms{V}} c_0}{2 \| u \|_{L^\infty}} = 0 \label{eq:HJ_hatur_geq} 
\ee
and a viscosity subsolution of
\be
\partial_t \hat u_r + \left \langle A_r x , \nabla_x \hat u_r \right \rangle + \frac{\lambda^q (2 \| u \|_{L^\infty})^{q-1}}{q} |P_0 \nabla_x \hat u_r |^q -\frac{f_+^{[r]}}{2 \| u \|_{L^\infty}} = 0  \label{eq:HJ_hatur_leq}
\ee

on the set $\ms{V}_r : = \widetilde D_{1/r}^{(\frac{1}{q})} \ms{V}$. We have
\be
\left \| \frac{f_+^{[r]}}{2 \| u \|_{L^\infty}} \right \|_{L^p(\ms{V}_r)} \leq  r^{ 1 - \frac{1}{p}  (\frac{N}{ q} + 1 + \sum_{i=1}^\kappa i n_i ) } \frac{\| f_+ \|_{L^p(\ms{V})}}{2 \| u \|_{L^\infty}} .
\ee
Moreover $\osc_{\ms{U}_r} \hat u_r \leq 1$.

Now choose $r>0$ small enough to satisfy the following:
\begin{align}
& r \leq h_\ast \\
& r \frac{\sup_{\ms{V}} c_0}{2 \| u \|_{L^\infty}} \leq \e_\ast \\
& r^{  1 - \frac{1}{p}  (\frac{N}{ q} + 1 + \sum_{j=1}^\kappa j n_j )  } \frac{ \| f_+ \|_{L^p(\ms{V})}}{2 \| u \|_{L^\infty}}  \leq \e_\ast .
\end{align}

The final hypothesis that we require in order to be able to apply Corollary~\ref{cor:HolderSmall} is that the set $K_r : = \widetilde D_{1/r}^{(\frac{1}{q})} K$ satisfies
\be \label{eq:CylinderCover}
\bigcup_{\xi \in K_r} l_{ \xi}^r \ov{Q}^r_1  \subset \ms{V}_r = \widetilde D_{1/r}^{(\frac{1}{q})} \ms{V}.
\ee
That is, we can place a cylinder at every point of $K_r$, and this cylinder will be fully contained in $\ms{V}_r$.
This can be ensured by making $r>0$ smaller (if necessary):
indeed, for all $\xi \in K_r$ we may write $\xi = \widetilde D_{1/r}^{(\frac{1}{q})} \xi '$ for some $\xi ' \in K$. Then by the properties of left translations \eqref{eq:LT_dil} and cylinders \eqref{eq:Cylinder_dil} with respect to dilations,
\be
\widetilde D_{r}^{(\frac{1}{q})} l_{ \xi}^r \ov{Q}^r_1 =  l_{ \xi '}^1 \widetilde D_{r}^{(\frac{1}{q})} \ov{Q}^r_1 = l_{ \xi '}^1 \ov{Q}^{1, \frac{1}{q}}_r .
\ee
Since $\lim_{r \to 0} \diam \ov{Q}^{1, \frac{1}{q}}_r = 0$, for all sufficiently small $r$ we have $l_{ \xi '}^1 \ov{Q}^{1, \frac{1}{q}}_r \subset \ms{V}$. Finally, the compactness of $K$ ensures that there exists a {\it uniform} $r>0$ such that $l_{ \xi '}^1 \ov{Q}^{1, \frac{1}{q}}_r \subset \ms{V}$ for all $\xi ' \in K$, and thus \eqref{eq:CylinderCover} holds.

We note that $r$ thus depends on the parameters $A, q, \Lambda, \lambda$ and $\| u \|_{L^\infty}$ (through $h_\ast$ and $\e_\ast$), as well as $\inf \{ |x-y| : x \in K, y \in\ms{V}^c\}$, $\sup_{\ms{V}} c_0$ and $\| f_+ \|_{L^p(\ms{V})}$.

Thus $\hat u_r$, $\ms{V}_r$ and $K_r$ satisfy the hypotheses of Corollary~\ref{cor:HolderSmall}. We deduce that, for all $(s,y), (t,x) \in K_r$ with $s \leq t$,
\be \label{eq:hatur_Holder}
|\hat u_r (s,y) - \hat u_r (t,x)| \leq  C \omega_\alpha \left( (t,x)^{-1}_r \diamond_r (s,y) \right) .
\ee
To conclude the proof, we rewrite  \eqref{eq:hatur_Holder} in terms of $u$.

Since $u = 2 \| u \|_{L^\infty(\ms{U})} \, \hat u_r \circ \widetilde D_{1/r}^{(\frac{1}{q})}$
\be
| u (s,y) -  u (t,x)| \leq  C \| u \|_{L^\infty(\ms{U})} \, \omega_\alpha \left( (\widetilde D_{1/r}^{(\frac{1}{q})} (t,x))^{-1}_r \diamond_r \widetilde D_{1/r}^{(\frac{1}{q})} (s,y) \right) .
\ee
Next, we identify $(\widetilde D_{1/r}^{(\frac{1}{q})} (t,x))^{-1}_r$: by definition
\be
(\widetilde D_{1/r}^{(\frac{1}{q})} (t,x))^{-1}_r \diamond_r  \widetilde D_{1/r}^{(\frac{1}{q})} (t,x) = (0,0) .
\ee
By \eqref{eq:LG_dil}, 
\be
 (\widetilde D_{1/r}^{(\frac{1}{q})} (t,x))^{-1}_r \diamond_r  \widetilde D_{1/r}^{(\frac{1}{q})} (t,x) =\widetilde D_{1/r}^{(\frac{1}{q})} \left [ \widetilde D_{r}^{(\frac{1}{q})} (\widetilde D_{1/r}^{(\frac{1}{q})} (t,x))^{-1}_r \diamond_1   (t,x) \right ] .
\ee

By \eqref{eq:Inverse_dil},
\be
\left( \widetilde D_{1/r}^{(\frac{1}{q})} (t,x) \right )^{-1}_r = \widetilde D_{1/r}^{(\frac{1}{q})} \, (t,x)_{1}^{-1}
\ee
Hence, applying \eqref{eq:LG_dil} once more,
\begin{align}
 (\widetilde D_{1/r}^{(\frac{1}{q})} (t,x))^{-1}_r \diamond_r \widetilde D_{1/r}^{(\frac{1}{q})} (s,y) & = \widetilde D_{1/r}^{(\frac{1}{q})} (t,x)^{-1}_1 \diamond_r \widetilde D_{1/r}^{(\frac{1}{q})} (s,y) \\
 & = \widetilde D_{1/r}^{(\frac{1}{q})} \left [ (t,x)^{-1}_1 \diamond_1 (s,y) \right ] .
\end{align}

Thus
\be
| u (s,y) -  u (t,x)| \leq  C_r \| u \|_{L^\infty(\ms{U})} \omega_\alpha ((t,x)^{-1}_1 \diamond_1 (s,y)) .
\ee

Finally, we observe that
\be
(t,x)^{-1}_1 \diamond_1 (s,y) = (s-t, y - e^{(s-t)A} x) .
\ee

\end{proof}

\bibliographystyle{abbrv} 
\bibliography{Hypo_MFG_bib}

\end{document}